\documentclass{article}

\RequirePackage[numbers]{natbib}
\RequirePackage[colorlinks,citecolor=blue,urlcolor=blue]{hyperref}
\RequirePackage{graphicx}
\usepackage{authblk}

\usepackage{amsthm,amsmath,amsfonts, amssymb}
\usepackage{color}
\usepackage[table]{xcolor}
\RequirePackage{graphicx}
\usepackage{algorithm}
\usepackage[noend]{algpseudocode}
\usepackage{dsfont}
\usepackage{subcaption}

\theoremstyle{plain}
\newtheorem{lemma}{Lemma}[section]
\newtheorem{prop}[lemma]{Proposition}

\newtheorem{theorem}[lemma]{Theorem}
\theoremstyle{definition}
\newtheorem{rk}[lemma]{Remark}
\newtheorem{assumption}{Assumption}
\theoremstyle{remark}

\newcommand{\cov}{\mathrm{cov}}

\renewcommand{\dim}{\mathrm{dim}}
\newcommand{\E}{\mathbb{E}}
\renewcommand{\Im}{\mathrm{Im}}
\renewcommand{\P}{\mathbb{P}}
\newcommand{\proj}{\mathrm{proj}}
\newcommand{\R}{\mathbb{R}}
\newcommand{\tr}{\text{\textnormal{tr}}}
\newcommand{\Var}{\mathbb{V}\text{\textnormal{ar}}}
\newcommand{\SA}{\mathfrak{S}_A}
\newcommand{\Sg}{\mathfrak{S}_g}

\newcommand{\Indicator}[1]{\mathds{1}{\left(#1\right)}}
\newcommand{\lin}{{\rm lin}}
\newcommand{\quadratic}{{\rm quad}}
\newcommand{\fin}{{\rm fin}}
\newcommand{\nfin}{{n_p}}
\newcommand{\iCE}{{\textnormal{\textrm{iCE}}}} %{\textnormal{\textrm{proj}}}
\newcommand{\npar}{{n}}
\newcommand{\one}{{\bf 1}}
\newcommand{\Proj}{{\textnormal{\textrm{Proj}}}} %{\textnormal{\textrm{proj}}}

\def\aux{g}

\title{Phase transition for conditional covariance matrices estimated by importance sampling, and implications for cross-entropy schemes in high dimension}
\author[1]{J.\ Beh}
\author[2]{J.\ Morio}
\author[3]{F.\ Simatos}
\affil[1]{{\small Computing \& Mathematical Sciences Division,  Mohamed Bin Zayed University of Artificial Intelligence,  Abu Dhabi, UAE}}
\affil[2]{{\small ONERA/DTIS, Université de Toulouse, France}}
\affil[3]{{\small Fédération ENAC ISAE-SUPAERO ONERA}}

\begin{document}

\maketitle

\begin{abstract}
	Motivated by recent high-dimensional reliability algorithms, we study a random matrix model corresponding to the estimation by importance sampling of $\Sigma_A$, the covariance matrix of a Gaussian vector conditioned by a rare event~$A$. For such models, we show that, in contrast to most random matrix models, the minimal sample size to ensure an accurate estimation of~$\Sigma_A$ needs to grow superlinearly in the dimension. More precisely, if $n$ is the sample size and~$d$ the dimension, we show that the estimation of $\Sigma_A$ is accurate if $n/d \gg L$ and not accurate if $n/d \ll L$ where $L$ is the maximum likelihood ratio. Moreover, we relate the asymptotic behavior of $L$ to $\lambda_1$, the smallest eigenvalue of the auxiliary covariance matrix, which suggests that $L$ (and thus the minimal sample size) will be larger for small values of $\lambda_1$. Numerical experiments confirm that this insight carries over to recent cross-entropy algorithms, with the best performing algorithms being also the ones with the largest smallest eigenvalue. Although backed up by our theoretical result, this result is surprising as these algorithms were not designed to act on the spectrum.
\end{abstract}

\section{Introduction} \label{sec:intro}

\subsection{Estimating conditional covariance matrices by importance sampling}

In this paper we consider a random matrix model corresponding to the estimation by importance sampling of the covariance matrix of the standard Gaussian distribution conditioned on some set $A$. Although seemingly simple, the naive Monte-Carlo method is inefficient for this task when the dimension is high and the set $A$ is rare and/or complex. This is typically the case in many rare event probability estimation problems where dedicated algorithms have been developed~\cite{El-Masri21-0, El-Masri24-0, Tong23-0, Uribe21-0, Wang16:0}. The random matrix model that we study here plays a key role in such modern high-dimensional reliability algorithms: we first state and discuss our main theoretical results concerning this random matrix model, and then investigate their implications for rare event algorithms in Section~\ref{sec:CE}.

In the rest of the paper, we consider $f(x) = (2\pi)^{-d/2} e^{-\lVert x \rVert^2/2}$ the density of the standard Gaussian distribution in dimension $d$, $A \subset \R^d$ measurable and $f|_A = f \xi_A / p$ the distribution $f$ conditioned on $A$, where $\xi_A$ is the indicator function of the set $A$ and $p = \P(X \in A)$ assumed to be strictly positive, with $X \sim f$. Let $\mu_A$ and $\Sigma_A$ be the mean and covariance matrix of~$f|_A$, respectively:
\[ \mu_A = \E(Y_A) \ \text{ and } \ \Sigma_A = \E( (Y_A - \mu_A) (Y_A - \mu_A)^\top), \ Y_A \sim f|_A. \]
We consider the case where $\Sigma_A$ is estimated by importance sampling using a Gaussian auxiliary distribution $N(\mu_\aux, \Sigma_\aux)$ with corresponding density $g$ as follows:
\begin{equation} \label{eq:hat-Sigma}
	\hat \Sigma_A = \frac{1}{n} \sum_{i=1}^n \ell_A(X_i) (X_i - \mu_A) (X_i - \mu_A)^\top \ \text{ with } \ \ell_A = \frac{f|_A}{g} \ \text{ and the $X_i$'s i.i.d. $\sim g$}.
\end{equation}
Note that $\E(\hat \Sigma_A) = \Sigma_A$, so that in fixed dimension $d$, the law of large numbers implies that~$\hat \Sigma_A$ is a consistent and unbiased estimator of $\Sigma_A$ as the sample size $n \to \infty$. However, we are interested in the accuracy of this estimation in the high-dimensional regime where both $d, n \to \infty$ and in particular in the minimal sample size $n$ ensuring, for large $d$, a correct estimation of $\Sigma_A$ by $\hat \Sigma_A$.

Since $\hat \Sigma_A$ can be written $\hat \Sigma_A = \frac{1}{n} SS^\top$ where $S$ is the rectangular matrix with columns $T_i = \ell_A(X_i)^{1/2} (X_i - \mu_A)$, $\hat \Sigma_A$ belongs to the class of empirical covariance matrices, which has been extensively studied. For instance, it has been first investigated assuming that entries of $S$ are i.i.d., with in particular the seminal works of Bai, Yin and co-authors~\cite{Bai93-0, Bai88-0, Yin88-0} identifying that, in the linear regime $n \propto d$ and in the isotropic case $\E(T_i T_i^\top) = I$, a fourth moment assumption on the entries of $S$ is necessary and sufficient for the error $\hat \Sigma_A - I$ to remain bounded. These works have triggered an intense activity to identify conditions under which the minimal sample size guaranteeing an accurate estimation of $I$ remains linear in the dimension. For instance, a sample size $O(d \log d)$ was shown in~\cite{Rudelson99-0} to be sufficient when the norm $\lVert T_i \rVert$ of the columns is almost surely $O(\sqrt d)$. The $\log d$ factor was shown to be superfluous when the columns are sub-exponential in~\cite{Adamczak10-0}, a result that was later relaxed to a $2+\varepsilon$ moment assumption in~\cite{Srivastava13-0}. Beyond the isotropic case, the minimal sample size was related in heavy tail settings to the effective rank of~$\Sigma$ in~\cite{Abdalla26-0, Jirak25-0}.

However, none of these results does apply to the model~\eqref{eq:hat-Sigma} because of the likelihood ratios $\ell_A(X_i)$ which, since they are scalar multipliers of the columns, induce what Srivastava and Vershynin refer to as a \textit{``dimension-free tail decay''} of the norm of the $T_i$'s~\cite{Srivastava13-0}. To our knowledge, such behavior has only been considered in~\cite[Remark~$4.9$]{Adamczak10-0} and~\cite[Section~$1.8$]{Srivastava13-0} but always as a counter-example to discuss the necessity of assumptions to ensure that the minimal sample size grows linearly in $d$. Such models therefore lie at the border of the current random matrix literature and, as far as we know of, they have not received significant attention so far, probably precisely because they are considered as outliers. To illustrate this, it is insightful to consider the following bound obtained in~\cite{Tikhomirov17-0} under proper moment assumptions:
\[ \frac{\lVert \hat \Sigma_A - \Sigma_A \rVert}{\lVert \Sigma_A \rVert} \leq F(d/n) + \frac{1}{n} \max_{i \leq n} \langle T_i, \Sigma_A^{-1} T_i \rangle. \]
(Here we state this bound in a very loose form in order to highlight the key behavior in which we are interested: later on we will use this result rigorously, see~\eqref{eq:bound-Tikhomirov}.) The function $F$ appearing in this expression vanishes at $0$, suggesting that the relative error can be made arbitrarily small in the linear regime provided that the max term is negligible compared to the $F$ term: and indeed, the author discusses conditions under which this occurs, leading to the conclusion in the abstract that ``\textit{the sample covariance matrix approximates the actual covariance matrix in the operator norm with any given precision as long as the size of the sample is a large multiple of $d$}''.

However, what we show here is that for models such as~\eqref{eq:hat-Sigma}, it is precisely the max term that dominates. Moreover, we derive a corresponding lower bound which allows us to conclude that, under appropriate technical assumptions, the accuracy of $\hat \Sigma_A$ is governed by the asymptotic behavior of the maximum likelihood ratio $\max_{i \leq n} \ell_A(X_i)$: more precisely, we show that $n/d$ needs to increase faster than $\max_{i \leq n} \ell_A(X_i)$ in order to ensure the accuracy of the approximation. In particular, it is not enough for $n$ to increase linearly in $d$ for $\hat \Sigma_A$ to accurately approximate $\Sigma_A$. The fact that the minimal sample size needs to grow superlinearly in the dimension has profound implications for recent high-dimensional rare event algorithms which are discussed in Section~\ref{sec:CE}. 

\subsection{Notation}

In the rest of the paper, all quantities implicitly depend on the underlying dimension $d$ but this is omitted from the notation. In particular, the sample size $n = n(d)$ is implicitly considered as a function of $d$: moreover, we always assume that $n = n(d) \to \infty$. For two real-valued sequences $u$ and $v$, we write $u \gg v$ or $v \ll u$ to say that $v/u \to 0$, and when we say that a sequence is bounded, it is when seeing this sequence as indexed by the dimension. Furthermore, we write $\Rightarrow$ convergence in distribution (as $d \to \infty$).

For a vector $x \in \R^d$, $\lVert x \rVert$ denotes its L\textsubscript2 norm. For a square matrix~$M$, we denote by $\lVert M \rVert$, $\tr M$ and $\lvert M \rvert$ its operator norm, its trace and its determinant, respectively, and when~$M$ is symmetric, we denote by $\lambda_{\min}(M)$ and $\lambda_{\max}(M)$ its smallest and largest eigenvalue, respectively. For scalar (possibly random) quantities $q_1, q_2$, we write 

In the sequel, $Y$, $Y_A$ and $X$ denote random variables with density $f$, $f|_A$ and $g$, respectively, and $X_i$'s denote i.i.d.\ random variables with density $g$ (recall that $g$ is the density of the auxiliary distribution used in the estimation of $\Sigma_A$ by importance sampling in~\eqref{eq:hat-Sigma}).

\subsection{Main results}

Our main result relies on three assumptions. Our first assumption is that $A$ can be described by a finite-dimensional space: this is a common assumption in the high-dimensional literature where high-dimensional sets of practical interest often have a lower dimensional structure, see for instance~\cite{Agapiou17-0}.

\begin{assumption} \label{ass:A}
	For each $d$, there exists a linear subspace $\SA$ such that $x \in A \Leftrightarrow P_{\SA} x \in A$. Moreover, $\dim(\SA)$ is bounded.
\end{assumption}

Assumption~\ref{ass:A} means that whether $x \in A$ or not only depends on a finite number of linear combinations of $x$'s coordinates: a simple yet insightful example is a half-plane $\{x: x^\top v \geq c\}$ for some $v \in \R^d$ and $c \in \R$. This assumption therefore captures the idea that $A$ can be described by a finite-dimensional space.

Assumption~\ref{ass:A} implies that $\Im(\Sigma_A - I) = \SA$ (see Lemma~\ref{lemma:mu-Sigma-A}): this means that $\Sigma_A$ follows the classical spiked models with the spikes spanning the set $\SA$, i.e., it can be written
\[ \Sigma_A = I + \sum_{i=1}^{\mathrm{dim}(\SA)} (l_i - 1) d_i d_i^\top, \]
with $l_i > 0$, $l_i \neq 1$ the eigenvalue of $\Sigma_A$ associated to $d_i$, the $d_i$'s orthonormal with $\mathrm{span}(d_i) = \SA$. For this reason, a recent stream of literature has proposed and studied importance sampling schemes where the auxiliary covariance matrix~$\Sigma_g$ also follows the spiked model\footnote{Other reasons motivate this choice, notably the fact that a low rank matrix involves less parameters to estimate than a full-rank one, which makes the estimation problem easier.}: as will be seen in  Section~\ref{sec:CE}, the following assumption is satisfied by recent high-dimensional reliability algorithms and therefore allows to cover a wide range of cases of practical interest.

\begin{assumption} \label{ass:g}
	For each $d$, the auxiliary distribution $g$ used in~\eqref{eq:hat-Sigma} is the density of the Gaussian distribution $N(\mu_g, \Sigma_g)$ with the following form:
	\[ \Sigma_g = I + \sum_{i=1}^r (\lambda_i - 1) v_i v_i^\top \]
	with $\lambda_1 \leq \cdots \leq \lambda_r$, all $\neq 1$, and the $v_i$'s orthonormal. Moreover, $\lambda_1 \in (0,1)$ does not depend on $d$ and $\lVert \mu_g \rVert$, $r$ and $\lambda_r$ are bounded.
\end{assumption}

\begin{rk}
	The assumption that $\lambda_1$ is independent of $d$ is made for simplicity: our results actually hold assuming that $\inf_d \lambda_1 > 0$.
\end{rk}

Note that, with this notation, we have $r = \mathrm{rank}(\Sigma_g - I)$, $1/\lambda_1 = \lVert \Sigma_g^{-1} \rVert$ and $\lambda_r = \max(\lVert \Sigma_g \rVert, 1)$. In the sequel, we define $\Sg = \mathrm{span}(v_1, \ldots, v_r) = \mathrm{Im}(\Sigma_g - I)$ the subspace spanned by the spikes $v_i$ of the auxiliary covariance matrix $\Sigma_g$. The two spaces $\SA$ and $\Sg$ play a key role: intuitively, the estimation $\hat \Sigma_A$ using $g$ as auxiliary distribution will be better if the spikes of the auxiliary covariance matrix $\Sigma_g$ are aligned with those of the sought conditional covariance matrix $\Sigma_A$: for these reasons, we will refer to the case $\Sg \subset \SA$ as the ``good'' case, and to the case $\Sg \subset {\SA}_\perp$ as the ``bad'' case.

Our last assumption assumes that $p$ remains bounded away from $0$.
\begin{assumption} \label{ass:p}
	$\inf_d p > 0$.
\end{assumption}

This assumption is crucial for our results. It is common in previous work studying the high-dimensional behavior of importance sampling, see the discussion in~\cite[Section $2.4$]{Beh25-0}. To our knowledge, the case $p \to 0$ is essentially unknown, and widely different behavior requiring significant different arguments could be needed to treat them.

Finally, to state our result we need the two orthogonal projections $P = P_{\SA \oplus \Sg \oplus \mathrm{span}(\mu_g)}$ and $P_\perp = P_{(\SA \oplus \Sg \oplus \mathrm{span}(\mu_g))_\perp}$ on the subspace $\SA \oplus \Sg \oplus \mathrm{span}(\mu_g)$ and its orthogonal, respectively, as well as (recall that the $X_i$'s are i.i.d.\ distributed according to~$g$)
\[ \hat 1 = \frac{1}{n} \sum_{i=1}^n \ell_A(X_i) \ \text{ and } \ \hat \mu_A = \frac{1}{n} \sum_{i=1}^n \ell_A(X_i) X_i. \]

\begin{theorem}\label{thm:main}
	Assume that Assumptions~\ref{ass:A}--\ref{ass:p} hold, and let
	\[ M = \frac{d}{n} \max_{i = 1, \ldots, n} \ell_A(X_i). \]
	Then for any $1 < a < \min(\frac{1}{1-\lambda_1}, 2)$ independent of $d$, the inequalities
	\begin{multline*}
		M - 1 \lesssim \frac{\lVert \hat \Sigma_A - \Sigma_A \rVert}{\lVert \Sigma_A \rVert} \lesssim \left( 1 + \hat 1 + \frac{\log n}{d} \right) M\\
		+ \lvert \hat 1 - 1 \rvert + \lVert P(\hat \mu_A - \mu_A) \rVert + \left( \frac{d}{n} \right)^{1-1/a} \log^4 \left( \frac{n}{d} \right)
	\end{multline*}
	hold with probability at least $1-\frac{4}{d}$. In particular, if $d \to \infty$ with $d \gg \log n$, then the following phase transition occurs:
	\begin{equation} \label{eq:PT}
		\frac{\lVert \hat \Sigma_A - \Sigma_A \rVert}{\lVert \Sigma_A \rVert} \Longrightarrow \left\{ \begin{array}{ll} 0 & \text{if } M \Rightarrow 0,\\+\infty & \text{if } M \Rightarrow +\infty. \end{array} \right.
	\end{equation}
\end{theorem}

In this statement, the constants hidden in the inequalities $\lesssim$ depend on the quantities appearing in the assumptions, namely $a$ and $\lambda_1$ as well as the supremum over $d$ of the quantities $\dim(\SA)$, $\mathrm{rank}(\Sigma_g - I)$, $\lVert \Sigma_g \rVert$, $\lVert \mu_g \rVert$ and $1/p$ (which are all assumed to be finite under Assumptions~\ref{ass:A}--\ref{ass:p}).

The lower bound of this result comes from an easy variational argument (see Section~\ref{subsub:lower-bound}): it is surprising that this simple argument yields the right behavior. The upper bound is less straightforward and relies on the non-asymptotic bound provided by Theorem~$1.1$ of~\cite{Tikhomirov17-0}. The main problem is that this theorem does not apply to $\hat \Sigma_A$ but to a related random model $\hat \Sigma$ because it requires centered columns. The main difficulties are therefore to control the error induced by considering $\hat \Sigma$ instead of $\hat \Sigma_A$ (Lemma~\ref{lemma:Sigma_A-Sigma}), and also to check that $\hat \Sigma$ satisfies the assumptions of~\cite[Theorem~$1.1$]{Tikhomirov17-0}, which essentially requires to control moments (Lemma~\ref{lemma:moments-S}).

The following lemma clarifies that the error terms $\lvert \hat 1 - 1 \rvert$ and $\lVert P(\hat \mu_A - \mu_A) \rVert$ (which come from the approximation of $\hat \Sigma_A$ by $\hat \Sigma$) are dimension-free: it could be improved, but it is enough for our purposes.

\begin{lemma} \label{lemma:error-terms}
	If Assumptions~\ref{ass:A}--\ref{ass:p} hold, then there exists $\gamma > 0$ such that $\lvert \hat 1 - 1 \rvert + \lVert P(\hat \mu_A - \mu_A) \rVert \lesssim \frac{1}{n^\gamma}$ with probability at least $1 - \frac{1}{n^\gamma}$.
\end{lemma}

The natural question in view of Theorem~\ref{thm:main} is therefore about the asymptotic behavior of the scaled maximum likelihood ratio $M$. The next result shows that $\lambda_1$ is an important quantity that, to some extent, governs whether $M \Rightarrow 0$ or $M \Rightarrow +\infty$. In particular, it suggests that $\hat \Sigma_A$ will better approximate $\Sigma_A$ when $\lambda_1$ is large: this observation constitutes the basis of the discussion of the implications of Theorem~\ref{thm:main} for recent high-dimensional reliability algorithms discussed in Section~\ref{sec:CE}.

\begin{prop} \label{prop:1/lambda}
	Assume that Assumptions~\ref{ass:A}--\ref{ass:p} hold. If $n \gg d^\kappa$ for some $\kappa > 1/\lambda_1$, then $M \Rightarrow 0$. Moreover, in the bad case $\Sg \subset {\SA}_\perp$, we have $M \Rightarrow \infty$ if $n \ll d^\kappa$ for some $\kappa < 1/\lambda_1$.
\end{prop}

This result therefore shows that, in the bad case $\Sg \subset {\SA}_\perp$, a phase transition occurs in the regime $d^{1/\lambda_1}$: loosely speaking, the approximation $\Sigma_A \approx \hat \Sigma_A$ will be accurate if $n \gg d^{1/{\lambda_1}}$ and not accurate if $n \ll d^{1/{\lambda_1}}$. However, although Theorem~\ref{thm:main} shows that a phase transition indeed occurs, it does not necessarily occur at the threshold $d^{1/{\lambda_1}}$: actually, the threshold at which the phase transition occurs depends on the interplay between $A$, $\SA$ and $\Sg$. To illustrate this we consider simple examples with different behavior. First, the following example shows that the phase transition can occur at $d^\kappa$ for any $1 < \kappa \leq 1/\lambda_1$.

\begin{prop} \label{prop:hyperplane-<}
	Fix some $\lambda_1 \in (0,1)$ and some $\kappa_* \in (1, \frac{1}{\lambda_1})$ independent of~$d$, and for each $d$, let:
	\begin{itemize}
		\item $A = \{x \in \R^d: \lvert \langle u, x \rangle \rvert \leq K \}$ where $u \in \R^d$ with $\lVert u \rVert = 1$, $K = \sqrt{2 \alpha \lambda_1 \log n}$ and $\alpha = \frac{1 - 1/\kappa_*}{1-\lambda_1}$;
		\item $\mu_g = 0$ and $\Sigma_g = I + (\lambda_1 - 1) u u^\top$.
	\end{itemize}
	Then Assumptions~\ref{ass:A}--\ref{ass:p} are satisfied with $\Sg = \SA = {\rm span}(u)$. Moreover, $M \Rightarrow \infty$ if $n \ll d^\kappa$ for some $\kappa < \kappa_*$ and $M \Rightarrow 0$ if $n \gg d^\kappa$ for some $\kappa > \kappa_*$.
\end{prop}

Of course, this result is very artificial as we have $p \to 1$, nonetheless it serves its purpose of showing that the phase transition can occur in the regime $d^{\kappa}$ for any $\kappa \in (1, 1/\lambda_1]$. In that respect, the following case is interesting for two reasons: first, it allows to consider any fixed $p \in (0,1)$ and moreover, it shows that even in the good case $\Sg = \SA$, the phase transition can occur at $d^{1/\lambda_1}$.

\begin{prop} \label{prop:hyperplane->}
	Fix some $\lambda_1 \in (0,1)$ and some $K > 0$ independent of $d$, and for each~$d$, let:
	\begin{itemize}
		\item $A = \{x \in \R^d: \langle u, x \rangle \geq K \}$ where $u \in \R^d$ with $\lVert u \rVert = 1$;
		\item $\mu_g = 0$ and $\Sigma_g = I + (\lambda_1 - 1) u u^\top$.
	\end{itemize}
	Then Assumptions~\ref{ass:A}--\ref{ass:p} are satisfied with $\Sg = \SA = {\rm span}(u)$. Moreover, $M \Rightarrow \infty$ if $n \ll d^\kappa$ for some $\kappa < 1/\lambda_1$.
\end{prop}

It is striking to realize that, depending on the set $A$, the good choice $\Sg \subset \SA$ may not improve the performance of the bad choice $\Sg \subset {\SA}_\perp$, with both (roughly) needing $n \gg d^{1/\lambda_1}$ to behave correctly.

\section{Efficiency of recent high-dimensional reliability algorithms: a new spectral interpretation} \label{sec:CE}

Before turning to the proofs of our main results in Section~\ref{sec:proofs}, we first explain in Sections~\ref{sub:CE} and~\ref{sub:justification} where do the random matrix model~\eqref{eq:hat-Sigma} and the Assumption~\ref{ass:g} come from. We then explain in Section~\ref{sub:implications} the implications of our results for recent high-dimensional reliability algorithms and illustrate them numerically.

\subsection{Motivation: cross-entropy in high dimension} \label{sub:CE}

Reliability algorithms are concerned with the problem of estimating the probability $p = \P(X \in F)$ of a rare (but important) event $F \subset \R^d$, usually called \textit{failure set}. When $F$ is a complex set for which it is costly to determine whether $x \in F$ or not, even when it is easy to sample from the distribution of $X$, it is well-known that the naive Monte Carlo method is not efficient, with a coefficient of variation given by $\sqrt{(1-p)/(np)}$ with $n$ the sample size. For instance, requiring a precision of $\varepsilon = 10\%$ when the probability is of the order of $p = 10^{-7}$ requires $n \approx 1/(p \varepsilon^2) \approx 10^9$ evaluations of whether $x \in F$ or not, which is prohibitive in many practical applications. Moreover, although this is not captured in this performance metric, the efficiency of the naive Monte Carlo method deteriorates in high dimension, so that tailored numerical schemes, such as adaptive splitting methods~\cite{Au01:0, Cerou12:0, Cerou19-0} or adaptive importance sampling schemes~\cite{Agapiou17-0}, are called upon. Here we focus on the cross-entropy method (CE)~\cite{Rubinstein04:0} (an adaptive importance sampling algorithm) and its recent variants~\cite{Papaioannou19-0}, especially the high-dimensional ones~\cite{El-Masri21-0, El-Masri24-0, Uribe21-0} which work for failure sets of the form $F = \{x \in \R^d: \varphi(x) \geq 0\}$ for some measurable function $\varphi: \R^d \to \R$. In this case, CE estimates $p$ by importance sampling with an auxiliary distribution~$g_p$, leading to estimates of the form
\begin{equation} \label{eq:IS}
	\hat p = \frac{1}{\nfin} \sum_{i=1}^\nfin \frac{f(Z_i)}{g_p(Z_i)} \xi_F(Z_i) \ \text{ with the $Z_i$'s i.i.d.\ $\sim g_p$.}
\end{equation}

The key idea of CE is to use as auxiliary distribution $g_p$ an approximation of the optimal Gaussian density $N(\mu_F, \Sigma_F)$ built iteratively by approaching the set $F = \{\varphi \geq 0\}$ by a sequence of simpler sets $F_t := \{\varphi \geq q_t\}$, with probabilities $\P(\varphi(X) \geq q_t)$ approaching $p$ (recall that for a general set $B$, $\mu_B$ and $\Sigma_B$ denote the mean and variance of $f|_B$). A deterministic version of CE is presented in Algorithm~\ref{alg:CE-det} showing how the $q_t$'s are iteratively computed: at iteration $t$ with current auxiliary distribution $g_t$, $q_t$ is computed as the $\rho$-quantile of $\varphi(X)$ under~$g_t$ (in practical applications $\rho$ is typically set at $5$--$10\%$), and the next auxiliary distribution is $g_{t+1} = N(\mu_{F_t}, \Sigma_{F_t}) = N(\mu_{t+1}, \Sigma_{t+1})$.
\begin{algorithm}[t]
\caption{One iteration of the deterministic version of CE} \label{alg:CE-det}
\begin{algorithmic}
\Require $\rho \in (0,1)$ and, at step $t$, current auxiliary distribution $g_t = N(\mu_t, \Sigma_t)$
\State 1. compute $q_t$ such that $\P_{g_t}(\varphi(X) \geq q_t) = \rho$;
\State 2. compute $\mu_{t+1} = \mu_{F_t}$ and $\Sigma_{t+1} = \Sigma_{F_t}$ the mean and variance of $f$ conditioned on $F_t = \{\varphi \geq q_t\}$;
\State 3. iterate with $g_{t+1} = N(\mu_{t+1}, \Sigma_{t+1})$.
\end{algorithmic}
\end{algorithm}

This deterministic version of CE highlights the essence of CE schemes, but it is not tractable because it involves the unknown parameters $q_t$, $F_t$, $\mu_t$ and $\Sigma_t$. The true CE scheme, presented in Algorithm~\ref{alg:CE}, is simply a stochastic version of Algorithm~\ref{alg:CE-det} where these unknown parameters are replaced by the estimations $\hat q_t$, $\hat F_t$, $\hat \mu_t$ and $\hat \Sigma_t$, leading to auxiliary distributions $\hat g_t = N(\hat \mu_t, \hat \Sigma_t)$ thought of as estimations of the $g_t$'s.
\begin{algorithm}[t]
\caption{One iteration of CE} \label{alg:CE}
\begin{algorithmic}
\Require $\rho \in (0,1)$, sample sizes~$m$ and $\npar$ and, at step $t$, current auxiliary distribution $\hat g_t = N(\hat \mu_t, \hat \Sigma_t)$
\State 1. generate $Y_1, \ldots, Y_m$ i.i.d.\ according to $\hat g_t$, independently from the $Y$'s;
\State 2. rank the $Y$'s according to their values by $\varphi$: $\varphi(Y_{(1)}) \leq \cdots \leq \varphi(Y_{(m)})$ and define $\hat q_t = \varphi(Y_{(\lfloor (1-\rho) m\rfloor)})$ and $\hat F_t = \{x: \varphi(x) \geq \hat q_t\}$;
\State 3. generate $X_1, \ldots, X_\npar$ i.i.d.\ according to $\hat g_t$;
\State 4. let $\displaystyle \hat p_t = \frac{1}{\npar} \sum_{k=1}^\npar \frac{f(X_i)}{\hat g_t(X_i)} \xi_{\hat F_t}(X_i)$ and $\displaystyle \hat \ell = \frac{f \xi_{\hat F_t}}{\hat p_t \hat g_t}$ and compute
\begin{equation} \label{eq:Sigma-CE}
	\hat \mu_{t+1} = \frac{1}{\npar} \sum_{k=1}^{\npar} \hat \ell(X_i) X_i \ \text{ and } \ \hat \Sigma_{t+1} = \frac{1}{\npar} \sum_{k=1}^{\npar} \hat \ell(X_i) X_i X_i^\top - \hat \mu_{t+1} \hat \mu_{t+1}^\top
\end{equation}
\State 5. iterate with $\hat g_{t+1} = N(\hat \mu_{t+1}, \hat \Sigma_{t+1})$.
\end{algorithmic}
\end{algorithm}

Moreover, although this algorithm displays favorable performance in small dimension (typically, up to dimension $d \approx 10$), it quickly deteriorates as the dimension increases. The estimation step $\Sigma_{t+1} \approx \hat \Sigma_{t+1}$ is the bottleneck limiting performance in high dimension, and, in order to overcome this problem, a recent stream of literature has proposed to reduce the dimension of the problem by projecting on subspaces of small dimensions. The general idea presented in Algorithm~\ref{alg:CE-proj} is to replace $\hat \Sigma_t$ in step 4 of Algorithm~\ref{alg:CE} by a covariance matrix that only updates variance terms in a small number of directions: typically, this amounts to replacing $\hat \Sigma_t$ by $\Proj(\hat \Sigma_t, {\bf v})$ where we define the following projection operator with ${\bf v} = (v_1, \dots, v_r)$ an orthonormal family:
\begin{equation} \label{eq:Proj}
	\Proj(\Sigma, {\bf v}) = \sum_{k=1}^r (\lambda_k - 1) v_k v_k^\top + I \ \text{ with } \ \lambda_k = v_k^\top \Sigma v_k.
\end{equation}
In other words, recent algorithms seeking to make CE work in high dimension propose to impose a spiked structure on all intermediate auxiliary covariance matrices.
\begin{algorithm}[t]
\caption{One iteration of CE with projection}\label{alg:CE-proj}
\begin{algorithmic}
\Require $\rho \in (0,1)$, sample sizes $m$ and $\npar$ and, at step $t$, current auxiliary distribution $\hat g^\proj_t = N(\hat \mu^\proj_t, \hat \Sigma^\proj_t)$,
\State 1. generate $Y_1, \ldots, Y_m$ i.i.d.\ according to $\hat g_t^\proj$ and define $\hat q_t$ and $\hat F_t$ from the $Y_i$'s similarly as in step 2 of Algorithm~\ref{alg:CE};
\State 2. generate $X_1, \ldots, X_\npar$ i.i.d.\ according to $\hat g_t^\proj$, independently from the $Y$'s, and compute $\hat p_t$, $\hat \mu_{t+1}$ and $\hat \Sigma_{t+1}$ from the $X_i$'s similarly as in step 4 of Algorithm~\ref{alg:CE} but with $\hat g^\proj_t$ instead of $\hat g_t$;
\State 3. compute the orthonormal family ${\bf v} = (v_1, \dots, v_r)$;
\State 4. define
\begin{align}\label{eq:Sigma-proj}
\hat \mu^\proj_{t+1} = \hat \mu_{t+1} \text{ and } \hat \Sigma^\proj_{t+1} = \Proj(\hat \Sigma_{t+1}, {\bf v})
\end{align}
and iterate with $\hat g_{t+1}^\proj = N(\hat \mu^\proj_{t+1}, \hat \Sigma_{t+1}^\proj)$.
\end{algorithmic}
\end{algorithm}
This imposed structure finds two different motivations: first, there are less parameters to estimate, making the estimation less noisy and second, as already mentioned above, this additional structure on the covariance matrix is reasonable when the set $F$ itself has a low-dimensional structure. Various propositions have been made for the projection directions ${\bf v}$, see for instance~\cite{El-Masri21-0, El-Masri24-0, Uribe21-0} and the discussion in~\cite[Section $2.3$]{Beh25-0}. In practice this choice has an important impact on the accuracy of the resulting estimator, but, to our knowledge, no theoretical result to date has assessed its influence on the performance of the final estimator $\hat p$.

\subsection{Justification of the random matrix model~\eqref{eq:hat-Sigma} and of Assumptions~\ref{ass:A} and~\ref{ass:g}} \label{sub:justification}

What is important for our purpose is that the estimation of the conditional covariance matrix $\Sigma_{t+1} = \Sigma_{F_t}$ in step 4 of Algorithm~\ref{alg:CE} (Eqn.~\eqref{eq:Sigma-CE}) is performed by importance sampling from the auxiliary distribution $\hat g_t$: this is where the random matrix model~\eqref{eq:hat-Sigma} kicks in, since it corresponds exactly to this framework.

Since the goal is to estimate $\Sigma_{F_t}$, this corresponds, with the notation of Section~\ref{sec:intro}, to the case where $A = F_t$ and so a natural question is whether $F_t$ does satisfy Assumption~\ref{ass:A} or not. Actually, what happens is that in many cases (notably, in all test cases considered in previous work that we know of), the low-dimensional structure of $F$ comes from the fact that~$\varphi$ itself is low-dimensional, and typically satisfies $\varphi = \varphi \circ P_U$ for some low (or even finite) dimensional space $U$. In this case, we have
\[ x \in F_t \Leftrightarrow \varphi(x) \geq q_t \Leftrightarrow \varphi(P_U x) \geq q_t \Leftrightarrow P_U x \in F_t \]
so that Assumption~\ref{ass:A} is satisfied with $\SA = U$.

Moreover, the estimation of $\Sigma_{F_t}$ is performed by importance sampling with the auxiliary distribution $\hat g^\proj_t = N(\hat \mu^\proj_t, \hat \Sigma^\proj_t)$ so that this corresponds, still with the notation of Section~\ref{sec:intro}, to $\mu_g = \hat \mu^\proj_t$ and $\Sigma_g = \hat \Sigma^\proj_t$. But since $\hat \Sigma^\proj_t$ is forced to follow the spiked model because of the projection step~\eqref{eq:Sigma-proj} (step 4 of Algorithm~\ref{alg:CE-proj}), Assumption~\ref{ass:g} is also automatically enforced (up to the assumption on $\lVert \mu_g \rVert = \lVert \hat \mu^\proj_t \rVert$, but which was shown in a closely related case to be satisfied~\cite{Beh25-0}).

In summary, the random matrix model~\eqref{eq:hat-Sigma} and the Assumptions~\ref{ass:A}--\ref{ass:p} under which our main results hold apply to a wide range of high-dimensional reliability algorithms of practical interest.

\subsection{Implications of our theoretical results for high-dimensional CE schemes} \label{sub:implications}

\subsubsection{A new spectral interpretation}

Our main results suggest that when estimating a conditional covariance matrix $\Sigma_A$ by importance sampling with auxiliary distribution $N(\mu_g, \Sigma_g)$, it is preferable to have a large $\lambda_{\min}(\Sigma_g)$: indeed, the estimation will be accurate whenever $n \gg d^{1/{\lambda_{\min}(\Sigma_g)}}$, and this condition actually becomes necessary in the bad case $\Sg \subset {\SA}_\perp$ but even, sometimes, in the good case $\Sg = \SA$. Thus, a larger $\lambda_{\min}(\Sigma_g)$ seems to correspond to a smaller minimal sample size required to accurately estimate $\Sigma_A$, as to the very least it safeguards against a bad choice of spikes for the auxiliary covariance matrix. It is interesting to note the parallel between this insight on the spectrum of the auxiliary covariance matrix and the classical idea of Stein to shrink eigenvalues of empirical covariance matrices (see for instance~\cite{Stein86-0} or the introduction of~\cite{Ledoit22-0} for a recent account).

As explained above, this insight carries over to high-dimensional CE schemes by considering $\Sigma_g = \hat \Sigma^\proj_t$, which eventually suggests that CE schemes will behave better for larger~$\lambda_{\min}(\hat \Sigma^\proj_t)$. To a lesser extent, the assumption $\sup \lambda_{\max}(\Sigma_g) < \infty$ in Assumption~\ref{ass:g} also suggests that a small $\lambda_{\max}(\hat \Sigma^\proj_t)$ is beneficial. But in practice, projection schemes almost always only update variance terms in directions of small variance, typically leading to $\lambda_{\max}(\hat \Sigma^\proj_t) = 1$. 

On this background, several projection directions $\bf v$ have recently been proposed. Importantly, the rationale for this projection step is not to smooth the spectrum of the auxiliary covariance matrix, but rather to decrease the number of parameters to be estimated, thereby reducing noise while focusing on relevant parameters. Projection directions are then chosen according to some optimality criterion, typically minimizing the Kullback--Leibler divergence with $f|_A$.

In the following, we take a new look at these algorithms and pay special attention to the influence of the choice of projection directions on the spectrum of the estimated covariance matrices. Numerical results presented below show that projection directions that work well are also the ones with the largest smallest eigenvalue $\lambda_{\min}(\hat \Sigma^\proj_t)$. This is quite surprising since, as explained above, these algorithms were not designed to act on the spectrum. Nonetheless, this gives credit to our main result, Theorem~\ref{thm:main}, for the understanding of the performance of importance sampling schemes, and opens the way for the design of efficient importance sampling schemes in high dimensions aiming to act on the spectrum of the auxiliary covariance matrices.

\subsubsection{Numerical set-up}

To investigate numerically the efficiency of high-dimensional CE schemes with respect to the behavior of the smallest eigenvalues $\lambda_{\min}(\hat \Sigma^\proj_t)$, we consider two CE schemes, two choices of projection directions and three functions $\varphi$:
\begin{algorithm}[t]
\caption{One iteration of improved CE (iCE) with projection}\label{alg:iCE-proj}
\begin{algorithmic}
\Require $\delta > 0$, sample sizes $m$ and $\npar$ and, at step $t$, current auxiliary distribution $\hat g^\iCE_t = N(\hat \mu^\proj_t, \hat \Sigma^\proj_t)$ and an additional parameter $\hat \sigma_t$
\State 1. generate $Y_1, \ldots, Y_m$ i.i.d.\ according to $\hat g^\iCE_t$ and compute
\[ \hat \sigma_{t+1} = \arg \min_{\sigma \in (0, \hat \sigma_t)} (\hat \delta_t(\sigma) - \delta )^2, \]
where $\hat \delta_t(\sigma)$ is the empirical coefficient of variation of $F_N(\varphi(Y_i)/\sigma) \frac{f(Y_i)}{\hat g^\iCE_t(Y_i)}$:
\begin{equation}\label{eq:iCE-delta}
    \hat \delta_t(\sigma) = \sqrt{ m \sum_{i=1}^m \left( \frac{f(Y_i)}{\hat g^\iCE_t(Y_i)} F_N\left( \frac{\varphi(Y_i)}{\sigma} \right) \right)^2} \biggl/ \sum_{i=1}^m \frac{f(Y_i)}{\hat g^\iCE_t(Y_i)} F_N \left( \frac{\varphi(Y_i)}{\sigma} \right);
\end{equation}
\State 2. generate $X_1, \ldots, X_\npar$ i.i.d.\ according to $\hat g^\iCE_t$, independently from the $Y_k$'s, and compute $\hat \mu_{t+1}$ and $\hat \Sigma_{t+1}$ from the $X$'s similarly as in~\eqref{eq:Sigma-CE} with
\[ \hat \ell = \frac{f F_N(\varphi(\cdot)/\hat \sigma_{t+1})}{\hat E \hat g^\iCE_t} \ \text{ where } \hat E = \frac{1}{\npar} \sum_{i=1}^\npar \frac{f(X_i)}{\hat g^\iCE_t(X_i)} F_N(\varphi(X_i) / \hat \sigma_{t+1}); \]
\State 3. define $\hat \mu^\proj_{t+1} = \hat \mu_{t+1}$ and compute $ \hat \Sigma^\proj_{t+1}$ from $\hat \Sigma_{t+1}$ similarly as in step 3 of Algorithm~\ref{alg:CE-proj};
\State 4. iterate with $\hat g^\iCE_{t+1} = N(\hat \mu^\iCE_{t+1}, \hat \Sigma^\proj_{t+1})$.
\end{algorithmic}
\end{algorithm}
\begin{description}
	\item[\textbf{CE schemes:}] we consider two CE schemes: the standard CE scheme (Algorithm~\ref{alg:CE}) and also the improved CE scheme (iCE) proposed in~\cite{Papaioannou19-0}. The complete description of iCE (with projection) is provided in Algorithm~\ref{alg:iCE-proj} but the main idea is quite simple: it amounts to replacing the indicator~$\xi_F$ by a smooth approximation, which is the purpose of the terms $F_N(\varphi(\cdot) / \sigma)$ in Algorithm~\ref{alg:iCE-proj}, which satisfy $F_N(\varphi(\cdot) / \sigma) \approx \xi_F$ for small $\sigma$'s;
	\item[\textbf{Projection directions:}] for CE and iCE, we will compare the algorithm without and with projection for two choices of projection directions. Each time, we consider for simplicity only one direction ($r = 1$ with the notation of Assumption~\ref{ass:g}). During iteration $t$, in step~3 of Algorithms~\ref{alg:CE-proj} and~\ref{alg:iCE-proj}, we will consider projecting on $v_1$ given by either 1/ the eigenvector of $\hat \Sigma_{t+1}$ associated to its smallest eigenvalue or 2/ the current mean $\hat \mu_{t+1}$. The first choice is a simplified version of the algorithm proposed in~\cite{El-Masri24-0}, while the second one was proposed in~\cite{El-Masri21-0}. This leads to three variants of CE and iCE denoted by CE, CE-eig and CE-mean for CE and iCE, iCE-eig and iCE-mean for iCE;
	\item[\textbf{Functions $\varphi$:}] we will test these six CE schemes on three different test functions, namely a linear function $\varphi_{\lin}(x) = x^\top \one - 5$ with $\one = \frac{1}{\sqrt d} (1, \ldots, 1)$, a quadratic function
	\[ \varphi_{\quadratic}(x) = x^\top \one - 4 - 1.25 (x(1) - x(2))^2 \]
	and a function stemming from applications in finance~\cite{Bassamboo08-0}:
	\[ \varphi_{\fin}(x) = \sum_{j=3}^d \Indicator{\phi(x(1),x(2),x(j)) \geq 0.5 \sqrt{d}} - 0.25 d - 0.1 \]
	where for any $(x, y, z) \in \R^3$,
	\[\phi(x,y,x) = \left(0.25 \, x + 3(1 - 0.25^2)^{1/2} z \right) \sqrt{ F_{\Gamma(6,6)}^{-1}\left(F_{N}(y) \right)} \]
	with $F_{\Gamma(6,6)}$ the cumulative distribution function of the Gamma distribution with both shape and scale parameters set to $6$.
\end{description}

The linear and quadratic test cases are standard functions considered in reliability analysis~\cite{El-Masri21-0, Papaioannou16-0, Papakonstantinou23-0, Uribe21-0}. The financial function $\varphi_\fin$ is for instance considered in~\cite{Bassamboo08-0, Chan12-0, El-Masri21-0}: it brings additional insight compared to $\varphi_\lin$ and $\varphi_\quadratic$ because it is not differentiable. The two functions $\varphi_\lin$ and $\varphi_\quadratic$ have an intrinsic dimension equal to one and two, respectively, and thus the corresponding failure sets satisfy Assumption~\ref{ass:A}. Formally, $\varphi_\fin$ does not satisfy this assumption, but a closer look suggests that it does satisfy it approximately: indeed, if the components $X(i)$ are i.i.d., then the law of large numbers suggests that
\[ \sum_{j=3}^d \Indicator{\phi(X(1),X(2),X(j)) \geq 0.5 \sqrt{d}} \approx d h(X(1), X(2)) \]
with $h(x,y) = \P(\phi(x,y, X(1)) \geq 0.5 \sqrt d)$, suggesting that $\varphi_\fin$ approximately has an intrinsic dimension of (at most) two.

For each test function $\varphi$, different CE schemes lead to different estimators $\hat p$ of the form~\eqref{eq:IS}. Recall that $\nfin$ is the sample size used in this final estimation step, while $n$ refers to the sample size used to learn the auxiliary distribution in step~$2$ of Algorithms~\ref{alg:CE-proj} (CE with projection) and~\ref{alg:iCE-proj} (iCE with projection). To assess the accuracy of each estimator, we consider the distribution of the relative error $\lvert \hat p - p \rvert / p$. This unknown distribution is estimated by kernel regression (displayed as so-called violin plots) obtained by repeating each estimation $N = 200$ times. Table~\ref{tab:parameters} gathers the relevant parameters for the numerical simulations presented below.
\begin{table}[t]
	\begin{center}
	\begin{tabular}{|c|c|c|c|}
		\hline
		\rowcolor{black!15!white} Test function & $\varphi_\lin$ & $\varphi_\quadratic$ & $\varphi_\fin$ \\\hline
		Dimension $d$ & $100$ & $334$ & $334$ \\\hline
		Probability $p$ & $2.9 \cdot 10^{-7}$ & $6.6 \cdot 10^{-6}$ & $1.8 \cdot 10^{-6}$ \\\hline
		Learning sample size $n$ & $10000$ & $5000$ & $5000$ \\\hline
		Final sample size $\nfin$ & \multicolumn{3}{c|}{$2000$} \\\hline
		Repetition $N$ & \multicolumn{3}{c|}{$200$} \\\hline
	\end{tabular}
	\end{center}
	\caption{Parameters used for the numerical simulations depending on the test function considered.}
	\label{tab:parameters}
\end{table}

\subsubsection{Numerical results for CE}

Figures~\ref{fig:CE} and~\ref{fig:iCE} compare the performance of CE, CE-eig and CE-mean (Figure~\ref{fig:CE}) and of iCE, iCE-eig and iCE-mean (Figure~\ref{fig:iCE}). The goal on these two figures is to check the relevance of Theorem~\ref{thm:main} for true CE schemes and also to assess the influence of the choice of the projection directions. Let us first comment results for CE on Figure~\ref{fig:CE}.

\begin{figure}[t]
		    \centering
		    \hspace*{0.1\linewidth}\includegraphics[width=0.5\linewidth]{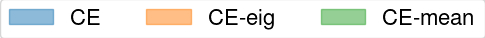} \newline
		    \begin{subfigure}[c]{.33\textwidth}
		        \centering
		        \includegraphics[width=\linewidth]{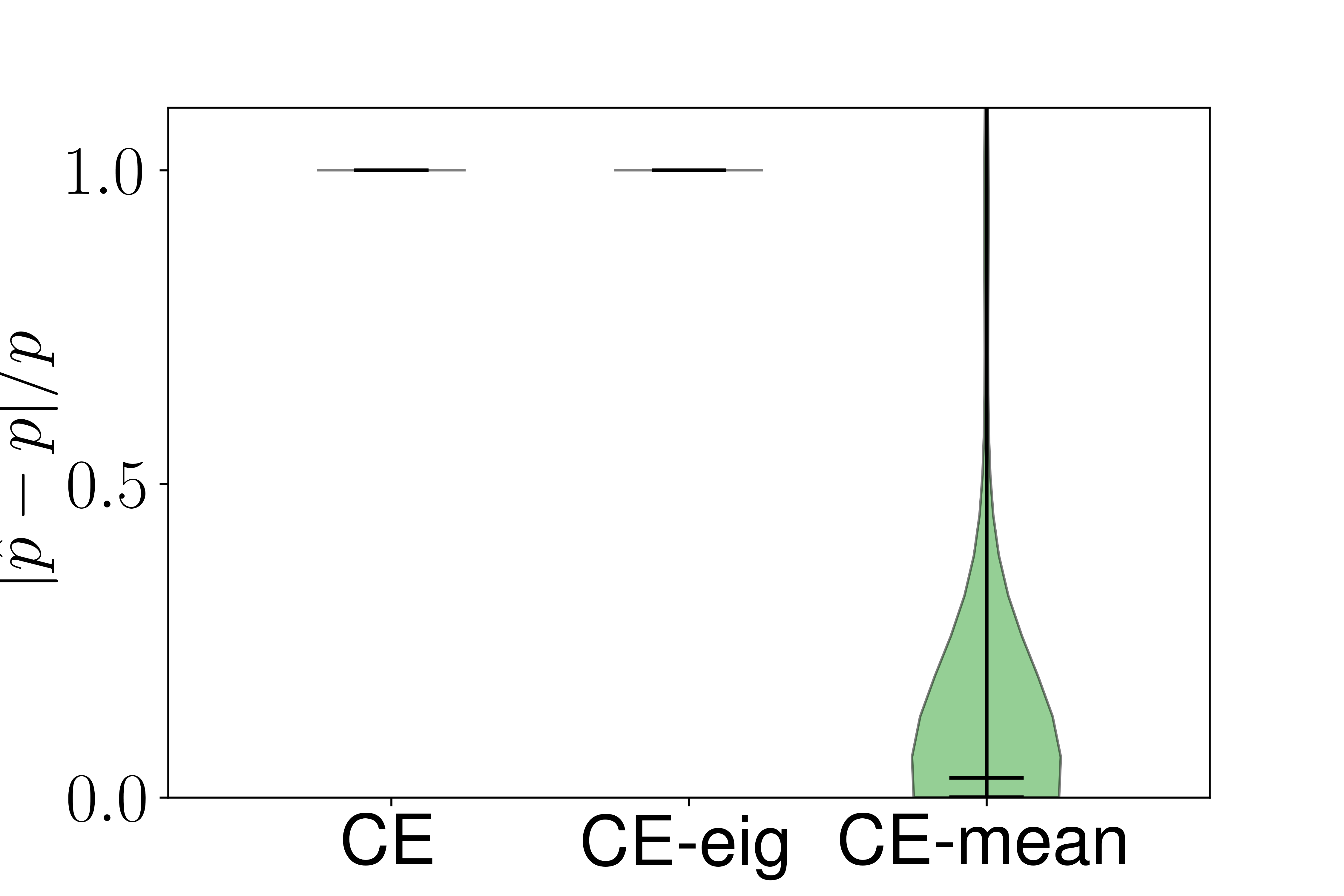}
		        \includegraphics[width=\linewidth]{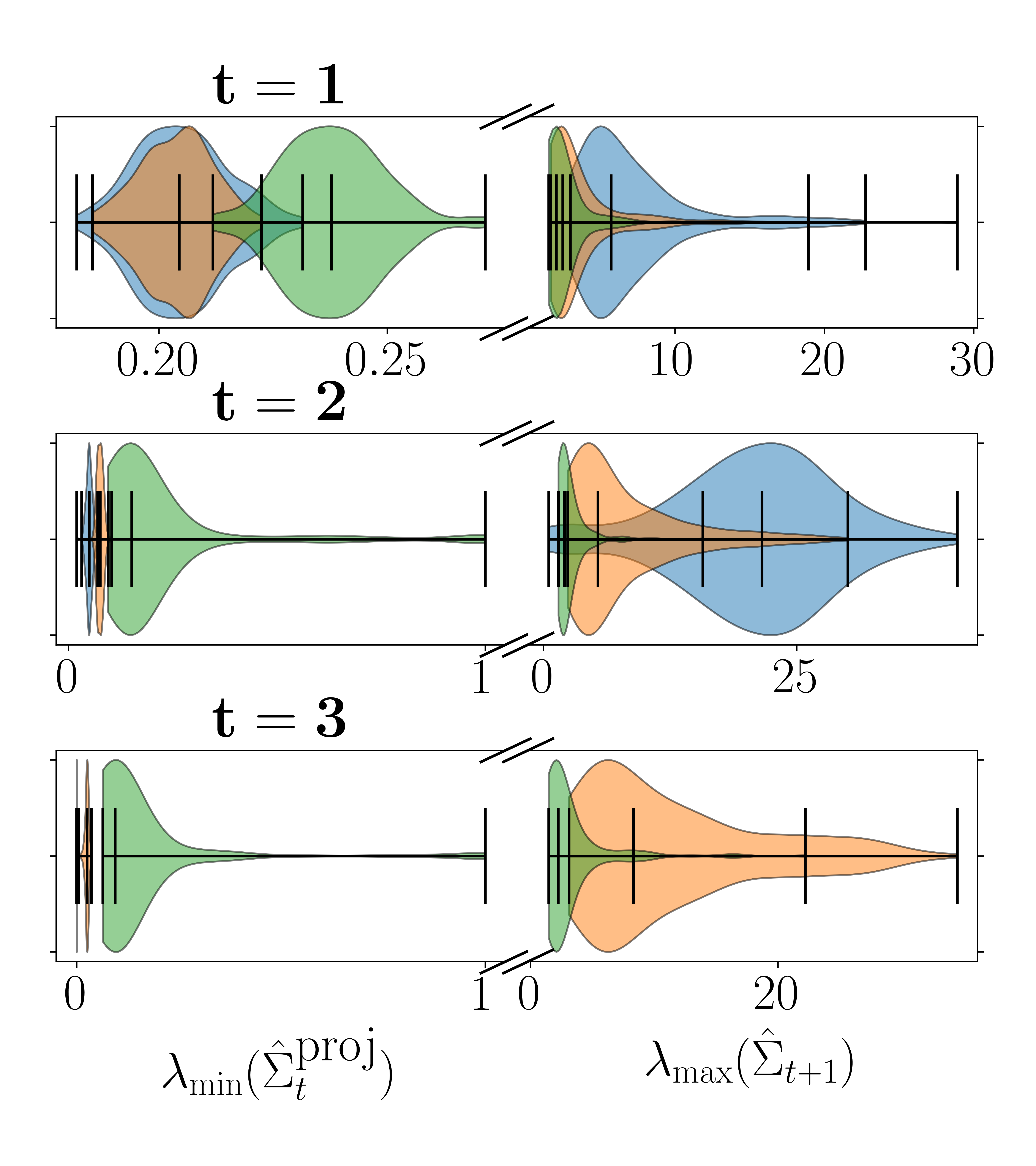}
		        \caption{Results for $\varphi_\lin$.}
		\label{subfig:CE-lin}
		    \end{subfigure}%
		    \hfill
		    \begin{subfigure}[c]{.33\textwidth}
		        \centering
		        \includegraphics[width=\linewidth]{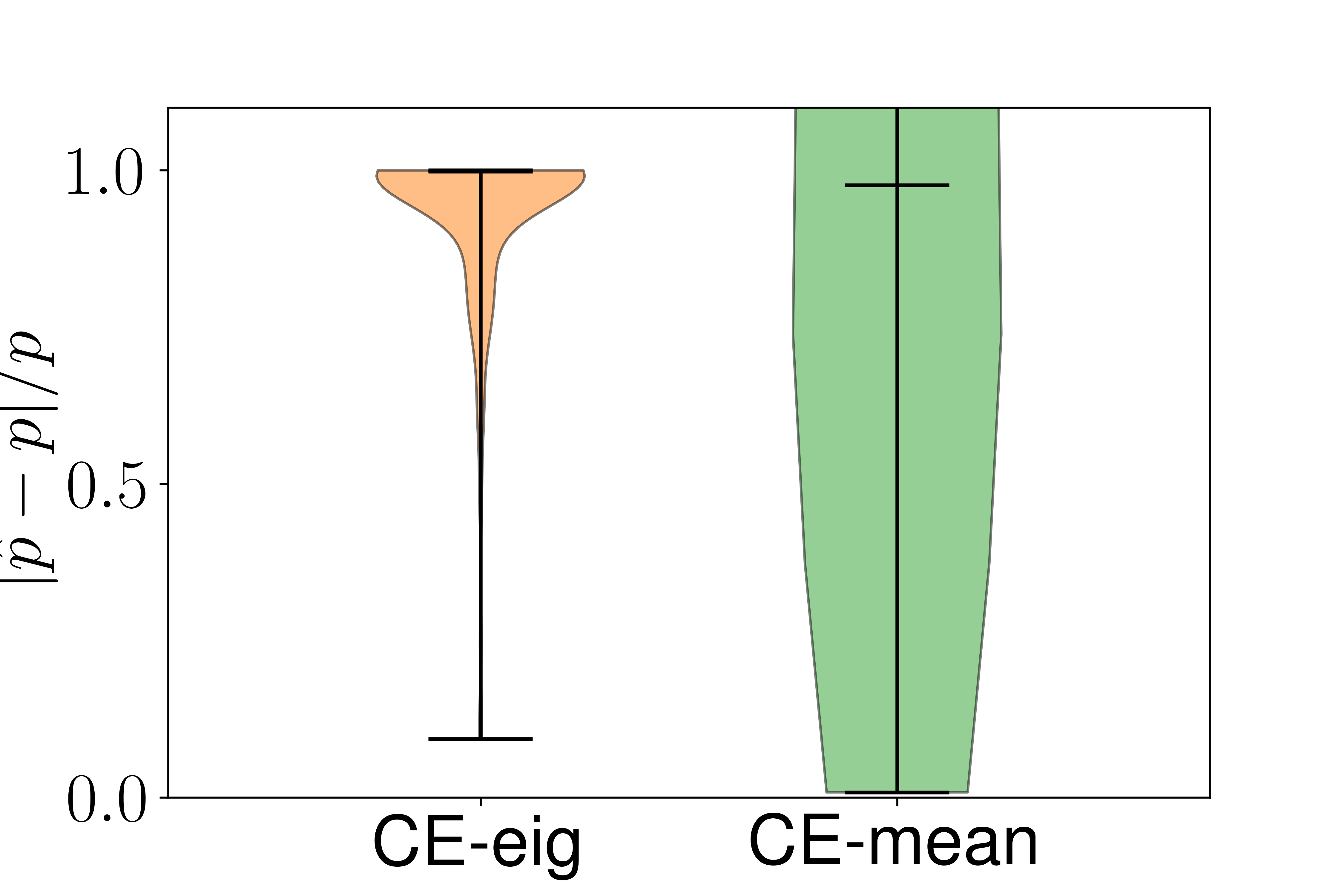}
		        \includegraphics[width=\linewidth]{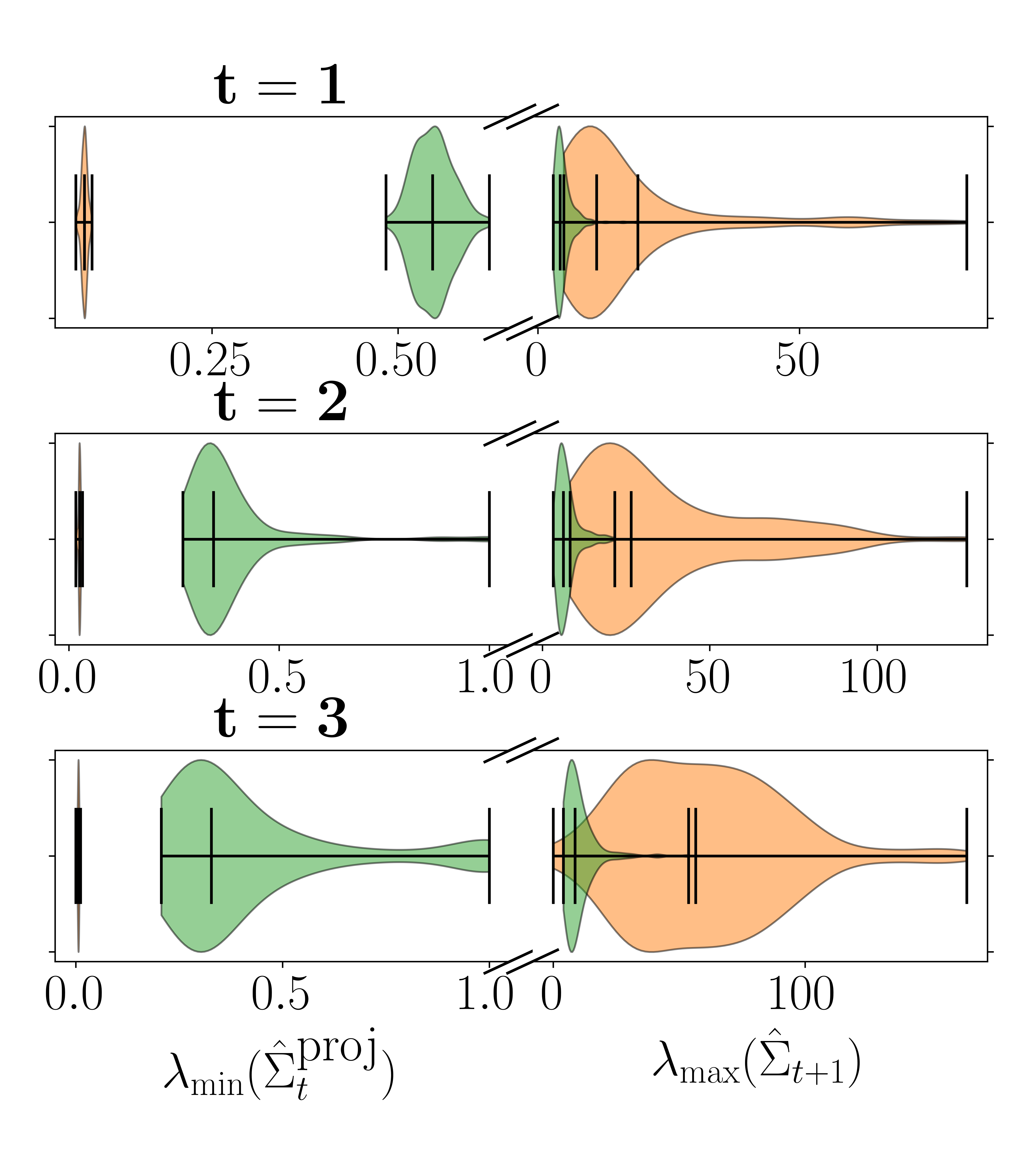}
		        \caption{Results for $\varphi_\quadratic$.}
		\label{subfig:CE-quad}
		    \end{subfigure}%
		    \hfill
		    \begin{subfigure}[c]{.33\textwidth}
		        \centering
		        \includegraphics[width=\linewidth]{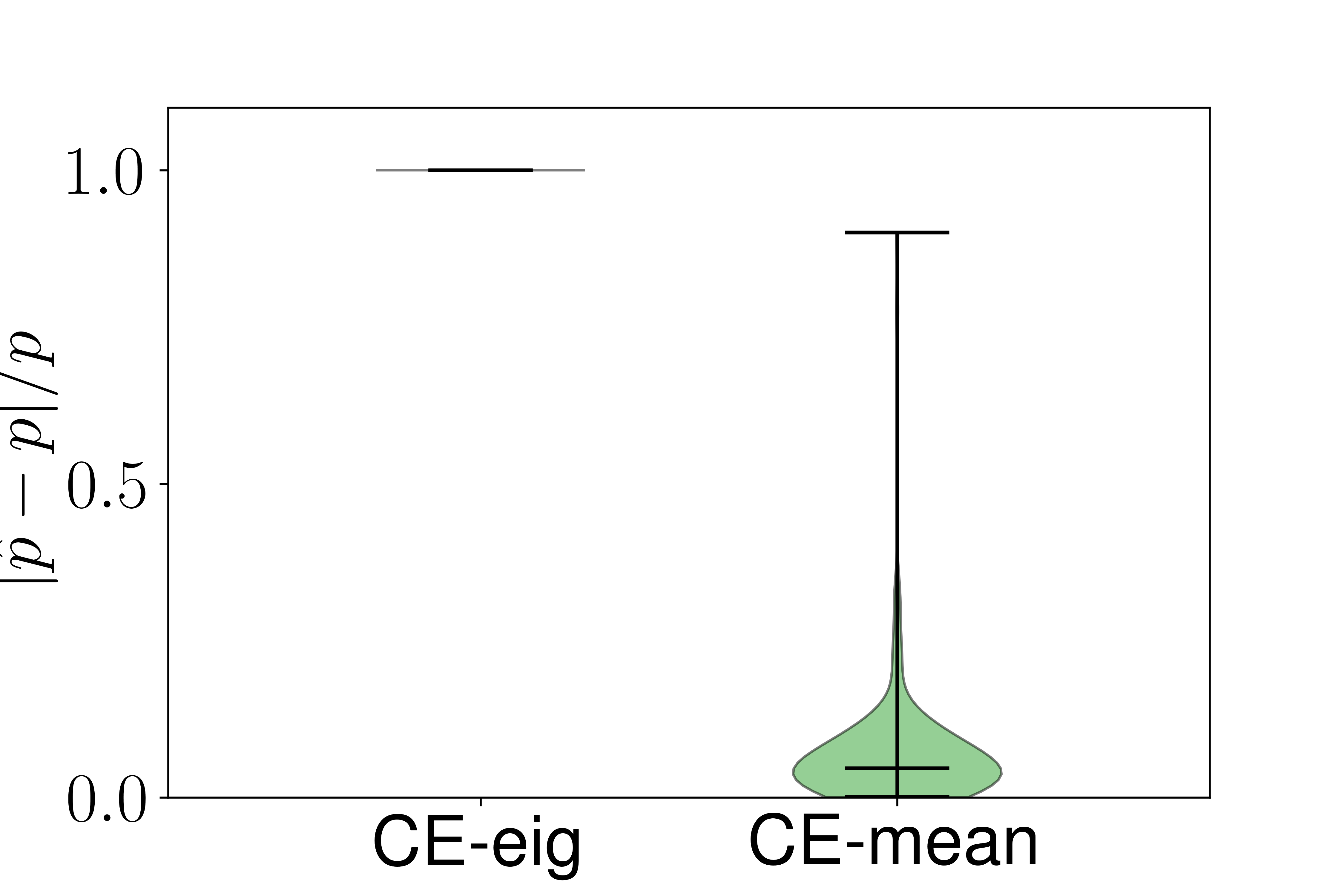}
		        \includegraphics[width=\linewidth]{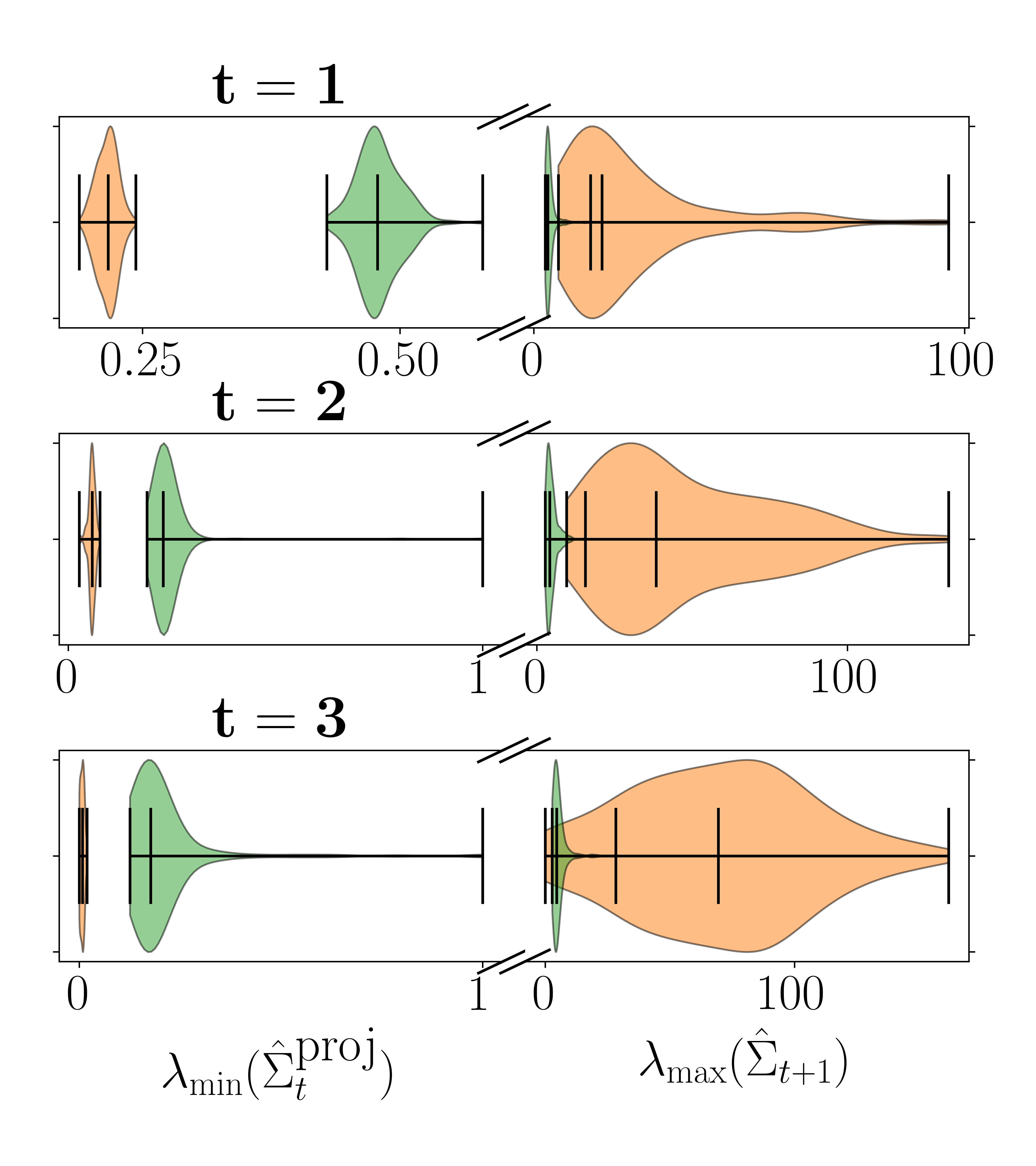}
		        \caption{Results for $\varphi_\fin$.}
		\label{subfig:CE-fin}
		    \end{subfigure}%
		    \hfill
	\caption{Interpretation of the performance of CE, CE-eig and CE-mean via the spectral behavior. Results for CE on the test functions $\varphi_\quadratic$ and $\varphi_\fin$ are not displayed because this scheme did not converge. For similar reasons, results for CE are not displayed for $\lambda_{\max}(\hat \Sigma_4)$ on the test function $\varphi_\lin$ since most repetitions diverge beyond the third iteration. Top figures display the distribution of the relative error $\lvert \hat p - p \rvert / p$. Bottom figures show the distribution of $\lambda_{\min}(\hat \Sigma^\proj_t)$ and $\lambda_{\max}(\hat \Sigma_{t+1})$ for the first three iterations $t = 1,2$ and $3$.}
    \label{fig:CE}
\end{figure}

The top row shows that CE-mean is the best algorithm, with the density of the relative error $\lvert \hat p - p \rvert / p$ more concentrated toward $0$. On the three test cases, CE and CE-eig yield very poor results: algorithms either do not converge, or yield a relative error close to $100\%$. In contrast, CE-mean gives much better results, and its performance for $\varphi_\lin$ and $\varphi_\fin$ is satisfactory. The quadratic case $\varphi_\quadratic$ remains challenging, but there is a substantial improvement from CE-eig to CE-mean. Lemma~\ref{lemma:mu-Sigma-A} gives interesting insight into why CE-mean performs well: actually, this lemma shows that $\mu_A \in \SA$, suggesting that using $\mu_A$ as projection direction is a good idea as it belongs to the subspace spanned by the spikes of $\Sigma_A$. Since $\hat \mu_{t+1}$ is meant to be close to $\mu_A$, this may explain the good performance of this choice.

The second row of Figure~\ref{fig:CE} gives an interpretation of these results by looking at the spectrum of the different covariance matrices during the first three iterations of the CE scheme. These results corroborate the previous discussion: we see that the best-performing algorithm, CE-mean, is also the one with the largest smallest eigenvalue $\lambda_{\min}(\hat \Sigma^\proj_t)$. Moreover, these numerical results suggest that the simplified model~\eqref{eq:hat-Sigma} yields relevant insight into CE schemes: indeed, we see that a small $\lambda_{\min}(\hat \Sigma^\proj_t)$ leads to a large $\lambda_{\max}(\hat \Sigma_{t+1})$, which is exactly the behavior pointed out in Theorem~\ref{thm:main}. The idea is that when the smallest eigenvalue of the covariance matrix of the auxiliary distribution is small, the sample size required for accurate estimation becomes very large. Although, as mentioned previously, the projection step typically only acts on small variance terms and sets $\lambda_{\max}(\hat \Sigma^\proj_{t+1}) = 1$, the fact that $\lambda_{\max}(\hat \Sigma_{t+1})$ is very large suggests that $\hat \Sigma_{t+1}$ does not satisfactorily estimate $\Sigma_{t+1}$. It is interesting to observe how this estimation error then propagates in the successive iterations: indeed, if $\Sigma_{t+1}$ is poorly estimated, then one may expect that $\hat \Sigma^\proj_{t+1}$ will also be rigged with noise. But as $\hat \Sigma^\proj_{t+1}$ is the covariance matrix of the next auxiliary distribution $\hat g_{t+1}$, this error is expected to propagate in the next iteration. And indeed, we clearly see in Figure~\ref{fig:CE} that poorly-performing schemes such as CE-eig have vanishing $\lambda_{\min}(\hat \Sigma^\proj_t)$ and diverging $\lambda_{\max}(\hat \Sigma_{t+1})$. Note that small $\lambda_{\min}$'s may seem like a good idea since the optimal covariance matrix $\Sigma_F$ has $\lambda_{\min}(\Sigma_F) \approx 0.0075$. Thus, the poor behavior of CE-eig may seem surprising at first glance, but Theorem~\ref{thm:main} sheds light on this phenomenon: a small $\lambda_{\min}$ gets closer to the optimal value $\lambda_{\min}(\Sigma_F)$ but also entails a more difficult estimation problem, requiring large sample size.

\subsubsection{Numerical results for iCE}

Figure~\ref{fig:iCE}
\begin{figure}[!t]
		    \centering
		    \hspace*{0.1\linewidth}\includegraphics[width=0.5\linewidth]{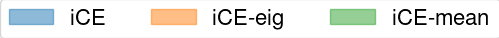} \newline
		    \begin{subfigure}[c]{.33\textwidth}
		        \centering
		        \includegraphics[width=\linewidth]{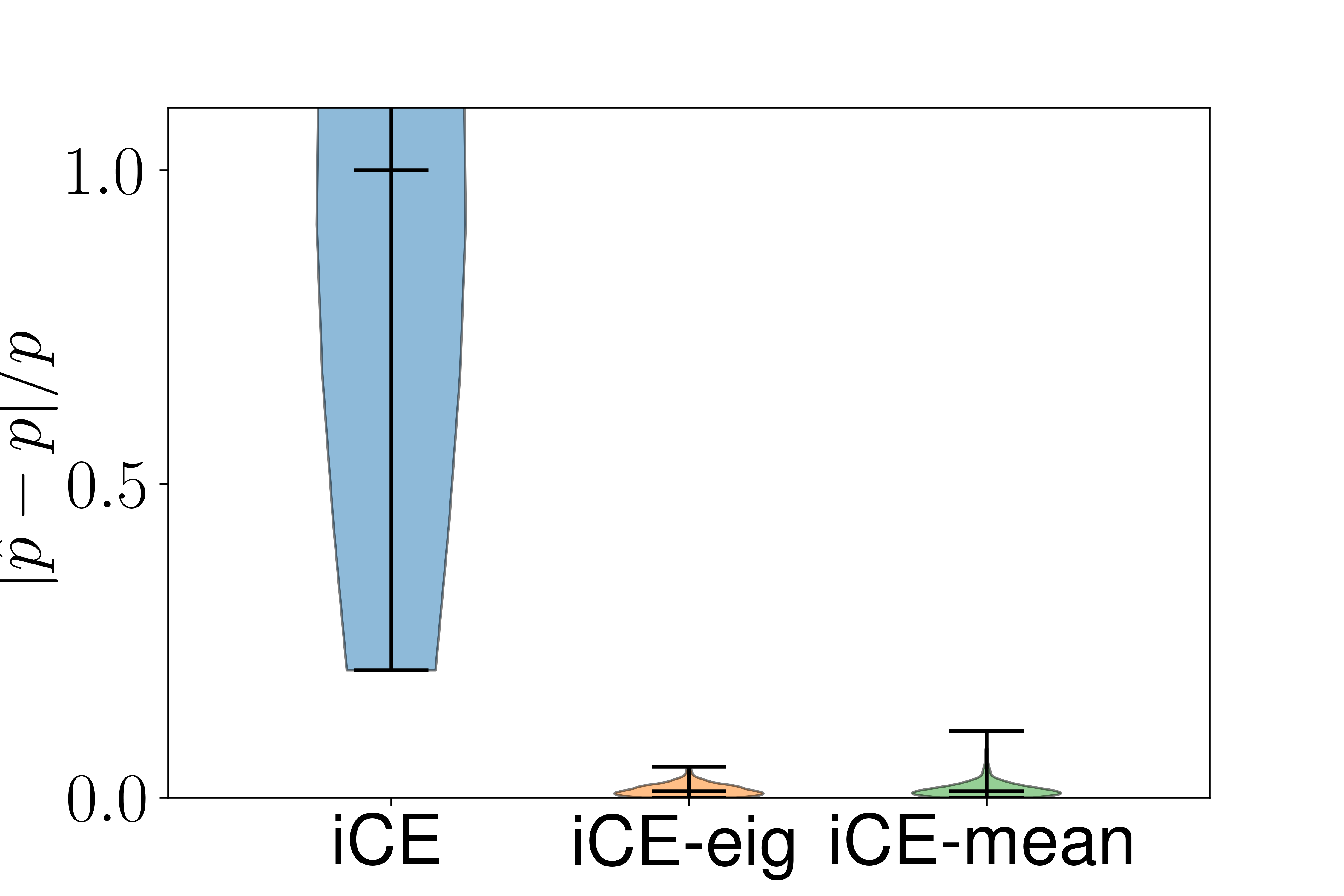}
		        \includegraphics[width=\linewidth]{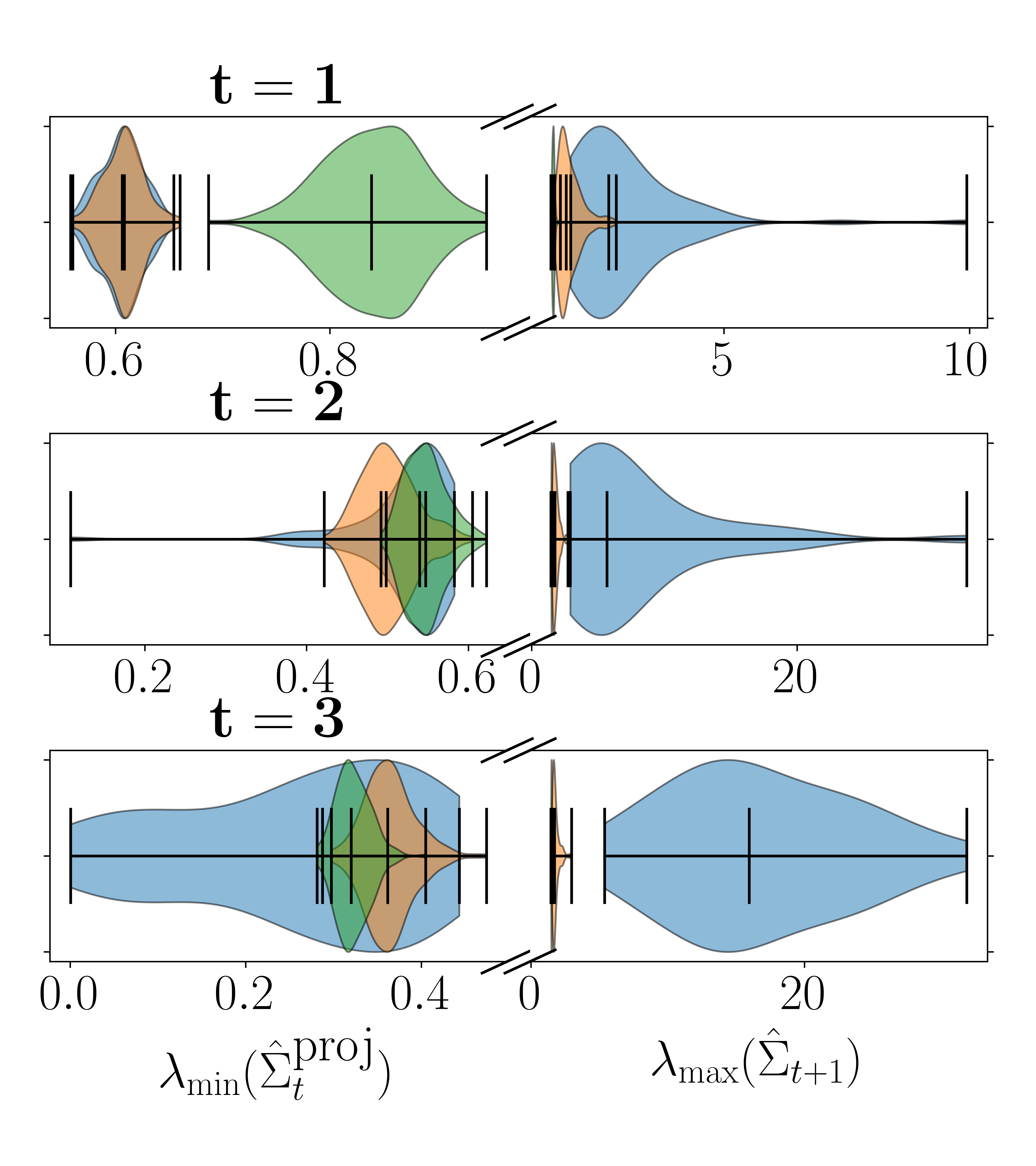}
		        \caption{Results for $\varphi_\lin$.}
		\label{subfig:iCE-lin}
		    \end{subfigure}%
		    \hfill
		    \begin{subfigure}[c]{.33\textwidth}
		        \centering
		        \includegraphics[width=\linewidth]{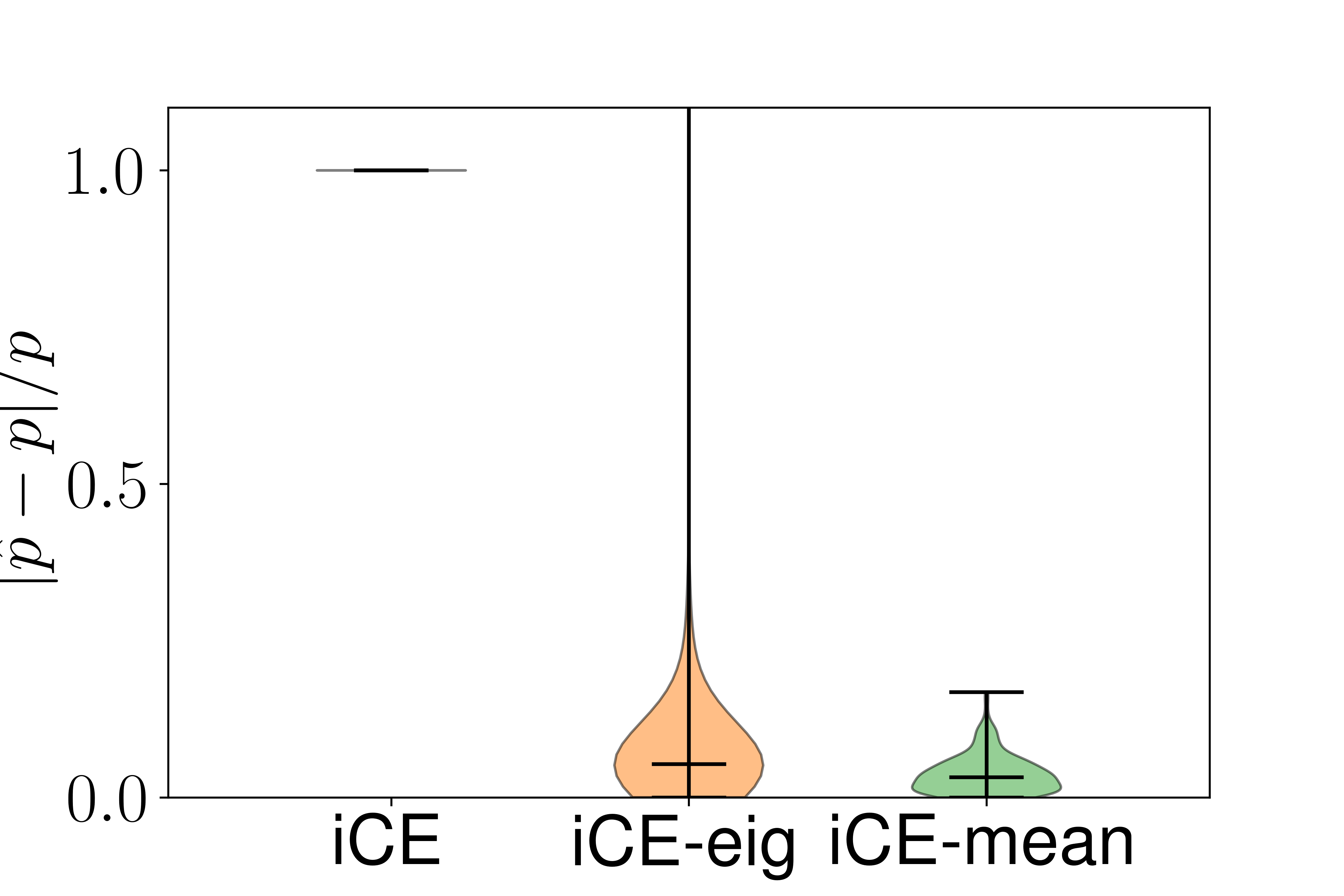}
		        \includegraphics[width=\linewidth]{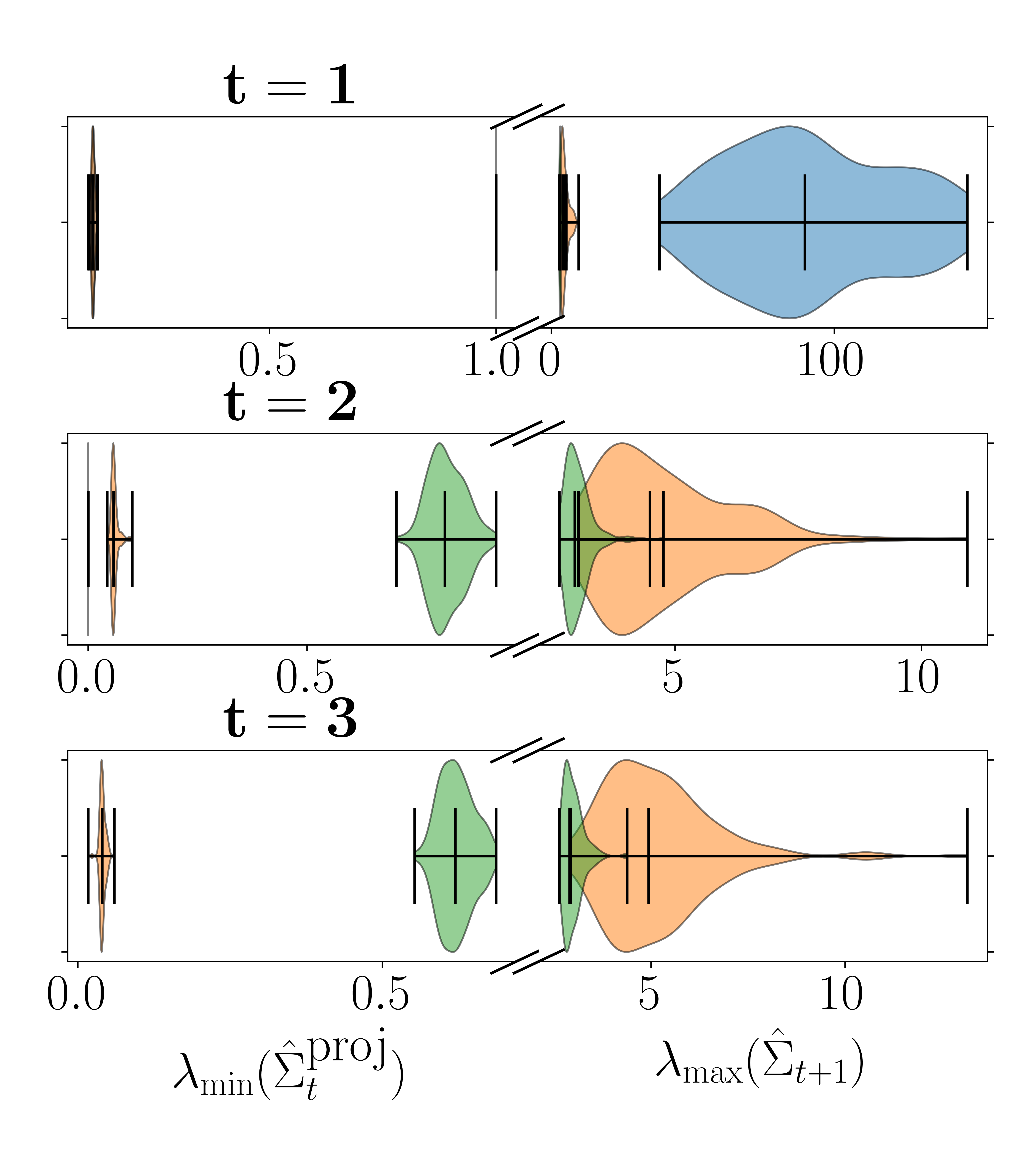}
		        \caption{Results for $\varphi_\quadratic$.}
		\label{subfig:iCE-quad}
		    \end{subfigure}%
		    \hfill
		    \begin{subfigure}[c]{.33\textwidth}
		        \centering
		        \includegraphics[width=\linewidth]{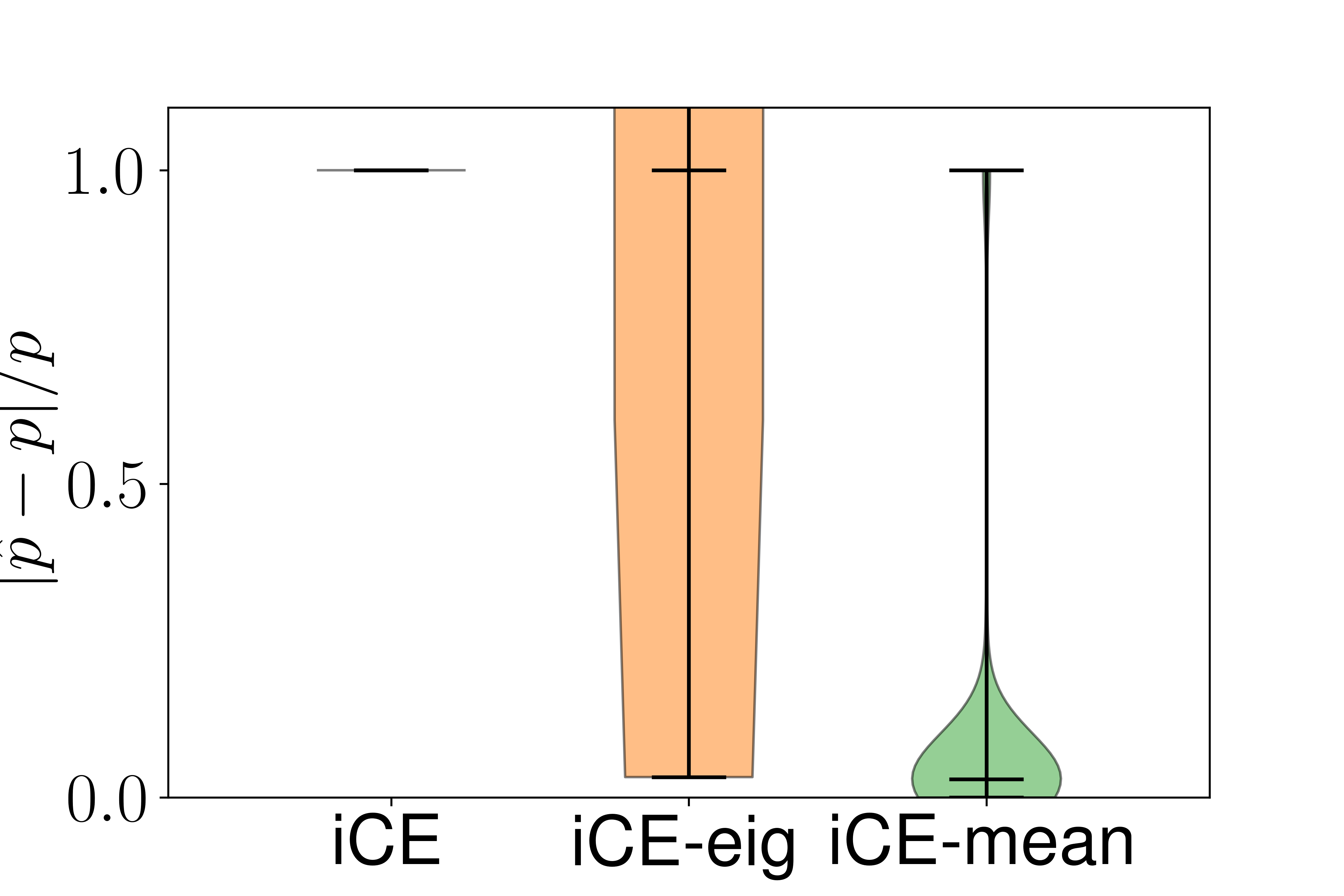}
		        \includegraphics[width=\linewidth]{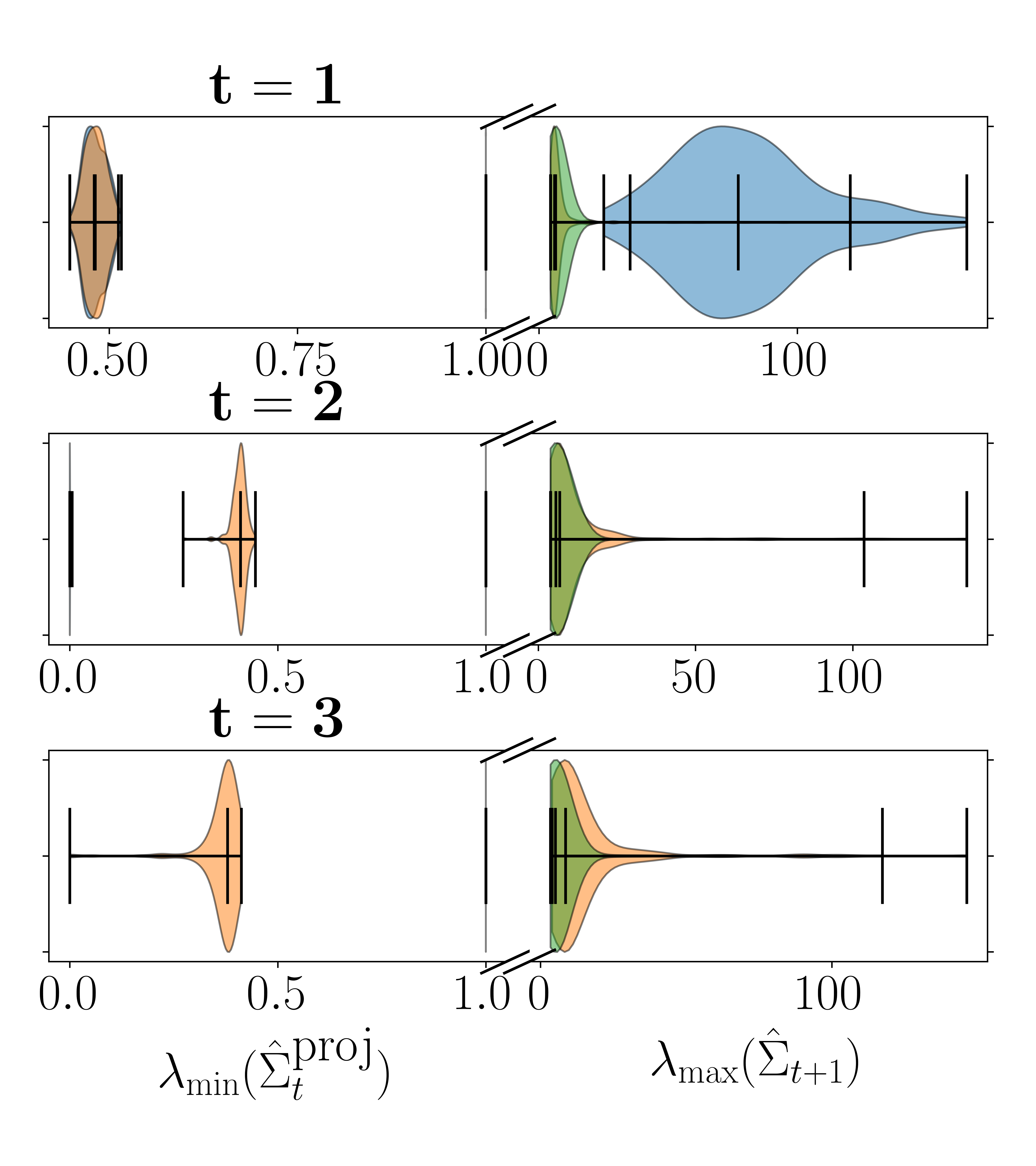}
		        \caption{Results for $\varphi_\fin$.}
		\label{subfig:iCE-fin}
		    \end{subfigure}%
		    \hfill
	\caption{Interpretation of the performance of iCE, iCE-eig and iCE-mean via the spectral behavior. Top figures display the distribution of the relative error $\lvert \hat p - p \rvert / p$. Bottom figures show the distribution of $\lambda_{\min}(\hat \Sigma^\proj_t)$ and $\lambda_{\max}(\hat \Sigma_{t+1})$ for the first three iterations $t = 1,2$ and~$3$.}
    \label{fig:iCE}
\end{figure}
shows that this insight is also confirmed by the behavior of iCE. We see on the top row of this figure that iCE behaves poorly, thereby justifying the need for projection. When considering projections, we see that iCE-mean offers significant improvements over iCE-eig for the two test cases $\varphi_\quadratic$ and $\varphi_\fin$ (Figures~\ref{subfig:iCE-quad} and~\ref{subfig:iCE-fin}). For these two examples, the same interpretation as for CE holds: iCE-mean is the variant with the largest $\lambda_{\min}(\hat \Sigma^\proj_t)$, and during the iterations of the different schemes, we see that a decreasing $\lambda_{\min}(\hat \Sigma^\proj_t)$ translates into an increasing $\lambda_{\max}(\hat \Sigma_{t+1})$. Finally, the linear test case (Figure~\ref{subfig:iCE-lin}) shows that the insight is quite sharp. In this case, iCE-eig and iCE-mean have similar median results, but the distribution of the relative error is slightly more concentrated for iCE-eig, which thus displays favorable performance over iCE-mean. Albeit subtle, this better performance is visible on the spectrum behavior, with iCE-eig eventually leading to a larger $\lambda_{\min}(\hat \Sigma^\proj_t)$ than iCE-mean.

Overall, these numerical results therefore confirm the relevance of Theorem~\ref{thm:main} for the understanding of the behavior of high-dimensional CE schemes.

\section{Proofs} \label{sec:proofs}

We now prove our theoretical results. We first establish preliminary results in Section~\ref{sub:preliminary}, then prove Theorem~\ref{thm:main} in Section~\ref{sub:proof-thm} and the remaining results in Sections~\ref{sub:proof-lemma} and~\ref{sub:proof-prop}.

\subsection{Preliminary results} \label{sub:preliminary}

In the rest of the paper we define $D(h_1, h_2) = \int h_1 \log \frac{h_1}{h_2}$ the Kullback--Leibler divergence between two densities $h_1$ and $h_2$ with $h_1$ absolutely continuous with respect to $h_2$, and we also define $\ell = f/g$ the likelihood ratio between $f$ and $g$, so that $\ell_A = \ell \xi_A / p$. Recall that $Y$, $Y_A$ and $X$ denote random variables with density $f$, $f|_A$ and $g$, respectively, and that $X_i$'s are i.i.d.\ random variables with density $g$.

First we check the claim made after the statement of Assumption~\ref{ass:A} that $\Sigma_A$ follows the spiked model, with the spikes belonging to $\SA$.

\begin{lemma}\label{lemma:mu-Sigma-A}
	If Assumption~\ref{ass:A} holds, then $\mu_A \in \SA$ and $\Im(\Sigma_A - I) \subset \SA$.
\end{lemma}

\begin{proof}
	We have to check that $P_{{\SA}_\perp} \mu_A = 0$ and that $P_{{\SA}_\perp} (\Sigma_A - I) = 0$, i.e., $P_{{\SA}_\perp} \Sigma_A = P_{{\SA}_\perp}$. By definition,
	\[ P_{{\SA}_\perp} \mu_A = \E(P_{{\SA}_\perp} Y \mid Y \in A) = \E(P_{{\SA}_\perp} Y \mid P_{\SA} Y \in A) = \E(P_{{\SA}_\perp} Y) = 0 \]
	using Assumption~\ref{ass:A} for the second equality and the independence between $P_{\SA} Y$ and $P_{{\SA}_\perp} Y$ for the third equality (since $Y$ is a standard Gaussian vector). Further,
	\[ P_{{\SA}_\perp} \Sigma_A  = P_{{\SA}_\perp} \E((Y - \mu_A) (Y - \mu_A)^\top \mid Y \in A) = \E( (P_{{\SA}_\perp} Y) (Y - \mu_A)^\top \mid P_{\SA} Y \in A) \]
	using Assumption~\ref{ass:A} for the second equality and the fact that $P_{{\SA}_\perp} \mu_A = 0$. Since $P_{{\SA}_\perp} Y$ and $P_{\SA} Y$ are independent and $\E P_{{\SA}_\perp} Y = 0$, we get
	\[ \E( (P_{{\SA}_\perp} Y) (Y - \mu_A)^\top \mid P_{\SA} Y \in A) = \E( (P_{{\SA}_\perp} Y) (P_{{\SA}_\perp}(Y - \mu_A))^\top \mid P_{\SA} Y \in A) \]
	and since $P_{{\SA}_\perp} \mu_A = 0$, we get
	\[ P_{{\SA}_\perp} \Sigma_A = \E( (P_{{\SA}_\perp} Y) (P_{{\SA}_\perp} Y)^\top \mid P_{\SA} Y \in A) = \E( (P_{{\SA}_\perp} Y) (P_{{\SA}_\perp} Y)^\top ) = P_{{\SA}_\perp} \]
	since $\E YY^\top = I$ and $P_{{\SA}_\perp}^2 = P_{{\SA}_\perp}$.
\end{proof}

\begin{lemma}\label{lemma:p}
	If Assumption~\ref{ass:p} holds, then $\lVert \mu_A \rVert$, $\lVert \Sigma_A \rVert$ and $\lVert \Sigma^{-1}_A \rVert$ are bounded. If in addition Assumption~\ref{ass:g} holds, then $D(f|_A, g)$ is also bounded.
\end{lemma}

\begin{proof}
	Let $g_A = N(\mu_A, \Sigma_A)$: it is not difficult to show that
	\[ D(f|_A, g_A) = -\log p - \Psi(\Sigma_A) - \frac{1}{2} \lVert \mu_A \rVert^2 \]
	and that
	\begin{equation} \label{eq:D}
		D(f|_A, g) = D(f|_A, g_A) + \Psi(\Sigma_g^{-1} \Sigma_A) + \frac{1}{2} (\mu_g - \mu_A)^\top \Sigma_g^{-1} (\mu_g - \mu_A),
	\end{equation}
	where $\Psi(\Sigma) = \frac{1}{2} (\tr(\Sigma) - \log \lvert \Sigma \rvert - d)$ is Stein's loss, see for instance~\cite[Lemma~$3.4$]{Beh25-0}. Since $D(f|_A, g_A) \geq 0$, this gives
	\[ \Psi(\Sigma_A) + \frac{1}{2} \lVert \mu_A \rVert^2 \leq - \log p. \]
	If $\psi(x) = \frac{1}{2} (x - \log x - 1)$, then we have $\Psi(\Sigma_A) = \sum \psi(\lambda)$ where the sum is taken over the eigenvalues of $\Sigma_A$. This implies that $\psi(\lambda_{\min}(\Sigma_A)), \psi(\lambda_{\max}(\Sigma_A))$ and $\lVert \mu_A \rVert$ are upper bounded by $-\log p$ which is bounded under Assumption~\ref{ass:A}. Since $\psi$ is decreasing on $(0,1]$ and increasing on $[1,\infty)$ with $\psi(x) \to \infty$ as $x \to 0+$ and $x \to +\infty$, the boundedness of $\psi(\lambda_{\min}(\Sigma_A))$ and $\psi(\lambda_{\max}(\Sigma_A))$ implies that of $1/\lambda_{\min}(\Sigma_A) = \lVert \Sigma_A^{-1} \rVert$ and of $\lambda_{\max}(\Sigma_A) = \lVert \Sigma_A \rVert$, which gives the desired boundedness of $\lVert \mu_A \rVert$, $\lVert \Sigma_A \rVert$ and $\lVert \Sigma_A^{-1} \rVert$.
	
	Let us now show that $D(f|_A, g)$ is bounded when Assumptions~\ref{ass:g} and~\ref{ass:p} hold. In view of~\eqref{eq:D} and the boundedness of $D(f|_A, g_A)$ it is enough to show that $\Psi(\Sigma_g^{-1} \Sigma_A)$ and $(\mu_g - \mu_A)^\top \Sigma_g^{-1} (\mu_g - \mu_A)$ are bounded. Since $(\mu_g - \mu_A)^\top \Sigma_g^{-1} (\mu_g - \mu_A) \leq \lVert \Sigma_g^{-1} \rVert \lVert \mu_g - \mu_A \rVert^2$ we get the boundedness of $(\mu_g - \mu_A)^\top \Sigma_g^{-1} (\mu_g - \mu_A)$ as a consequence of the boundedness of $\lVert \Sigma_g^{-1} \rVert$, $\lVert \mu_g \rVert$ (by Assumption~\ref{ass:g}) and $\lVert \mu_A \rVert$ (by Assumption~\ref{ass:A}, as we have just proved). Finally, the boundedness of $\Psi(\Sigma_g^{-1} \Sigma_A)$ follows from the boundedness of $\mathrm{rank}(M-I)$ and $\lVert M \rVert$ for $M = \Sigma_g, \Sigma^{-1}_g, \Sigma_A$ and $\Sigma^{-1}_A$, see for instance~\cite[Lemmas~$3.7$ and~$3.8$]{Beh25-0}.
\end{proof}

\begin{lemma} \label{lemma:moments-ell}
	Assume that Assumptions~\ref{ass:g} and~\ref{ass:p} hold. Then it holds that $\inf_d \E \ell_A(X)^{1/2} > 0$ as well as $\sup_d \E \ell_A(X)^{1+\varepsilon} < \infty$ for any fixed $0 < \varepsilon < \frac{\lambda_1}{1-\lambda_1}$.
\end{lemma}

\begin{proof}
	Let us show the result for $\E \ell_A(X)^{1/2}$: we have
	\[ \E \ell_A(X)^{1/2} = \E \left( \frac{f|_A(X)^{1/2}}{g(X)^{1/2}} \right) = \E \left( \frac{g(Y_A)^{1/2}}{f|_A(Y_A)^{1/2}} \right) = \E \left[ \exp \left( \frac{1}{2} \log \frac{g(Y_A)}{f|_A(Y_A)} \right) \right] \]
	and so Jensen's inequality and the convexity of the exponential function give
	\[ \E \ell_A(X)^{1/2} \geq \exp \left( \E \left( \frac{1}{2} \log \frac{g(Y_A)}{f|_A(Y_A)} \right) \right) = e^{-\frac{1}{2} D(f|_A, g)}. \]
	Since $D(f|_A, g)$ is bounded by Lemma~\ref{lemma:p}, we get that $\inf \E \ell_A(X)^{1/2} > 0$. Let us now control $\E \ell_A(X)^{1+\varepsilon}$. Recall that $\ell = f/g$: since $\inf p > 0$, we have
	\begin{equation} \label{eq:ell-ell_A}
		\E \ell_A(X)^{1+\varepsilon} = \E \ell_A(Y_A)^\varepsilon = \frac{1}{p^{1+\varepsilon}} \E (\ell(Y)^\varepsilon \xi_A(Y)) \lesssim \E \ell(Y)^\varepsilon.
	\end{equation}
	From the definitions of $f = N(0,I)$ and $g = N(\mu_g, \Sigma_g)$, we can compute
	\begin{equation} \label{eq:ell}
		\ell(x) = \lvert \Sigma_g \rvert^{1/2} e^{\frac{1}{2} \mu_g^\top \Sigma_g^{-1} \mu_g} \exp \left(-\frac{1}{2} (x^\top (I - \Sigma^{-1}_g) x + 2 \mu_g^\top \Sigma^{-1}_g x ) \right)
	\end{equation}
	which gives
	\begin{equation} \label{eq:ell-eps}
		\ell(x)^\varepsilon f(x) = \frac{1}{(2\pi)^{d/2}} \lvert \Sigma_g \rvert^{\varepsilon/2} e^{\frac{\varepsilon}{2} \mu_g^\top \Sigma_g^{-1} \mu_g} \exp \left(-\frac{1}{2} (x^\top V_\varepsilon x + 2 \varepsilon y^\top x ) \right)
	\end{equation}
	with $V_\varepsilon = \varepsilon (I - \Sigma_g^{-1}) + I$ and $y = \Sigma_g^{-1} \mu_g$. 
	%%%%%%%%%%%
	% DETAILS %
	%%%%%%%%%%%
	% \begin{align*}
	% 	\ell(x) & = \frac{ \frac{1}{(2\pi)^{d/2}} e^{- \frac{1}{2} \lVert x \rVert^2} }{ \frac{1}{(2\pi)^{d/2} \lvert \Sigma_g \rvert^{1/2}} e^{- \frac{1}{2} (x-\mu_g)^\top \Sigma_g^{-1} (x-\mu_g)} }\\
	% 	& = \lvert \Sigma_g \rvert^{1/2} \exp \left(-\frac{1}{2} (\lVert x \rVert^2 - (x-\mu_g)^\top \Sigma_g^{-1} (x-\mu_g)) \right)\\
	% 	& = \lvert \Sigma_g \rvert^{1/2} \exp \left(-\frac{1}{2} (x^\top (I - \Sigma^{-1}_g) x + 2 \mu_g^\top \Sigma_g^{-1} x - \mu_g^\top \Sigma_g^{-1} \mu_g ) \right)
	% \end{align*}
	%%%%%%%%%%%%%%%
	% END DETAILS %
	%%%%%%%%%%%%%%%
	Since $\lambda_{\min}(V_\varepsilon) = \varepsilon (1 - 1 / \lambda_1) + 1 > 0$ by choice of $\varepsilon$, we see that $V_\varepsilon$ is invertible and so we can use the identity
	\[ x^\top V_\varepsilon x + 2 \varepsilon y^\top x = (x + \varepsilon V_\varepsilon^{-1} y)^\top V_\varepsilon (x + \varepsilon V_\varepsilon^{-1} y) - \varepsilon^2 y^\top V^{-2}_\varepsilon y \]
	to get
	\[ \ell(x)^\varepsilon f(x) = \frac{1}{(2\pi)^{d/2}} \lvert \Sigma_g \rvert^{\varepsilon/2} e^{\frac{\varepsilon}{2} \mu_g^\top \Sigma_g^{-1} \mu_g + \frac{\varepsilon^2}{2} y^\top V^{-2}_\varepsilon y}\times \exp \left(-\frac{1}{2}(x + \varepsilon V_\varepsilon^{-1} y)^\top V_\varepsilon (x + \varepsilon V_\varepsilon^{-1} y) \right). \]
	Integrating, we obtain
	\[ \E \ell(Y)^\varepsilon = \int \ell^\varepsilon f = \lvert \Sigma_g \rvert^{\varepsilon/2} \lvert V_\varepsilon \rvert^{-1/2} e^{\frac{\varepsilon}{2} \mu_g^\top \Sigma_g^{-1} \mu_g + \frac{\varepsilon^2}{2} y^\top V^{-2}_\varepsilon y}. \]
	When Assumption~\ref{ass:g} holds, it can be easily checked that this quantity is bounded, which implies the boundedness of $\E \ell(Y)^\varepsilon$ and thus that of $\E \ell_A(X)^{1+\varepsilon}$ by~\eqref{eq:ell-ell_A}, which finally gives the desired result.
\end{proof}

\begin{lemma} \label{lemma:W}
	If Assumptions~\ref{ass:A} and~\ref{ass:g} hold, then $\ell_A = \ell_A \circ P$ and $PX$ and $P_\perp X$ are independent Gaussian vectors with $PX \sim N(\mu_g, \Sigma_g - P_\perp)$ and $P_\perp X \sim N(0, P_\perp)$.
\end{lemma}

\begin{proof}
	Let us prove that $\ell_A = \ell_A \circ P$. From~\eqref{eq:ell} we get
	\[ \ell(x) \propto \exp \left(-\frac{1}{2} (x^\top (I - \Sigma^{-1}_g) x - 2 \mu_g^\top (I - \Sigma^{-1}_g) x + 2 \mu_g^\top x ) \right). \]
	Assumption~\ref{ass:A} implies that $\xi_A(x) = \xi_A(Px)$ while Assumption~\ref{ass:g} implies that $(I-\Sigma^{-1})x = (I-\Sigma^{-1}) Px$ and $\mu_g^\top x = \mu_g^\top P x$ (recall that $P$ is the projection on $\SA \oplus \Sg \oplus \mathrm{span}(\mu_g)$). Thus we get $\ell_A = \ell_A \circ P$ as desired.
	
	Let us now show the result concerning the joint distribution of $PX$ and $P_\perp X$. Since $X$ is a Gaussian vector, so is $(PX, P_\perp X)$. Concerning the means, we have $\E PX = P \mu_g = \mu_g$ and $\E P_\perp X = P_\perp \mu_g = 0$. Concerning the variance, the definitions of $P_\perp$ and $\Sigma_g$ readily imply that $P_\perp (\Sigma_g - I) = 0$ and $P (\Sigma_g - I) = \Sigma_g - I$, and so
	\[ \Var(PX) = P \Sigma_g P = P (\Sigma_g - I) P + P = \Sigma_g - I + P = \Sigma_g - P_\perp \]
	and
	\[ \Var(P_\perp X) = P_\perp \Sigma_g P_\perp = P_\perp (\Sigma_g - I) P_\perp + P_\perp = P_\perp. \]
	It remains to check the independence. Since $(P X, P_\perp X)$ is Gaussian, we only have to check that the covariance is $0$: we have
	\[ \mathrm{Cov}(PX, P_\perp X) = P \Sigma_g P_\perp^\top = P (\Sigma_g - I) P^\top_\perp = 0 \]
	since we have seen that $P_\perp (\Sigma_g - I) = 0$.
\end{proof}

\subsection{Proof of Theorem~\ref{thm:main}} \label{sub:proof-thm}

We now prove Theorem~\ref{thm:main}, so in the rest of this section we assume that Assumptions~\ref{ass:A}--\ref{ass:p} hold. We establish the lower bound in Section~\ref{subsub:lower-bound} and the upper bound in Section~\ref{subsub:upper-bound}, and we then check in Section~\ref{subsub:phase-transition} that this implies the phase transition~\eqref{eq:PT}.

Note that the bounds in the statement are claimed to hold in an event of probability $1 - \frac{4}{d}$: the factor $4$ comes from the fact, in the proof below, we identify four events $E_1, E_2, E_3$ and $E_4$, each with probability at least $1-\frac{1}{d}$, and such that the bounds hold in the event $E_1 \cap E_2 \cap E_3 \cap E_4$ which is of probability at least $1 - \frac4d$. These events are highlighted as the proofs unfold in Remarks~\ref{rk:E1},~\ref{rk:E2} and~\ref{rk:E3-E4} below. They essentially correspond to events where the norm of a Gaussian vector is close to its mean: to control their probability, we will use the following classical concentration bounds for the norm of a standard Gaussian vector $G \sim N(0,I)$:
\begin{equation} \label{eq:concentration-G}
	\P(\lVert G \rVert \leq \E\lVert G \rVert - r) \leq e^{-\frac{1}{2}r^2} \ \text{ and } \ \P(\lVert G \rVert \geq \E\lVert G \rVert + r) \leq e^{-\frac{1}{2}r^2},
\end{equation}
see for instance~\cite[Eqn.~(2.8) p.\ 131]{Bakry12-0}.

\subsubsection{Lower bound} \label{subsub:lower-bound}

We have $\lVert \hat \Sigma_A - \Sigma_A \rVert \geq \lVert \hat \Sigma_A \rVert - \lVert \Sigma_A \rVert$ and if $J = \arg \max \ell_A(X_i)$, since $\lVert \hat \Sigma_A \rVert = \sup_{x: \lVert x \rVert = 1} x^\top \hat \Sigma_A x$ we have by considering $x = (X_J - \mu_A) / \lVert X_J - \mu_A \rVert$
\[ \lVert \hat \Sigma_A \rVert \geq \frac{1}{\lVert X_J - \mu_A\rVert^2} \frac{1}{n} \sum_{i=1}^n \ell_A(X_i) \langle X_J - \mu_A, X_i - \mu_A \rangle^2 \geq \frac{1}{n} \ell_A(X_J) \lVert X_J - \mu_A\rVert^2. \]
Since $P_\perp \mu_A = 0$ by Lemma~\ref{lemma:mu-Sigma-A}, we have $\lVert X_J - \mu_A\rVert \geq \lVert P_\perp X_J \rVert$ and so, since $\ell_A(X_J) = \max_{i \leq n} \ell_A(X_i) = (n/d) M$,
\[ \lVert \hat \Sigma_A - \Sigma_A \rVert \geq M \frac{\lVert P_\perp X_J \rVert^2}{d} - \lVert \Sigma_A \rVert \]
which gives
\[ \frac{\lVert \hat \Sigma_A - \Sigma_A \rVert}{\lVert \Sigma_A \rVert} \gtrsim M \frac{\lVert P_\perp X_J \rVert^2}{d} - 1 \]
by Lemma~\ref{lemma:p}. Since $\ell_A = \ell_A \circ P$ by Lemma~\ref{lemma:W}, $I$ only depends on the $P X_i$'s. Since moreover, still by Lemma~\ref{lemma:W}, the $PX_i$'s and the $P_\perp X_i$'s are independent, we get that $P_\perp X_J$ has the same distribution as $P_\perp X_i$ for any fixed, determinist index $i$, i.e., $P_\perp X_J \sim N(0, P_\perp)$ by Lemma~\ref{lemma:W}. Thus~\eqref{eq:concentration-G} gives $\lVert P_\perp X_J \rVert \geq \E \lVert N(0, P_\perp) \rVert - \sqrt{2 \log d}$ with probability at least $1-\frac{1}{d}$. Since it is easily seen that $\E \lVert N(0, P_\perp) \rVert - \sqrt{2 \log d} \gtrsim \sqrt d$, we get the desired result.

\begin{rk} \label{rk:E1}
	The first event that we need to consider is thus $E_1 = \{\lVert P_\perp X_J \rVert \geq \E \lVert N(0, P_\perp) \rVert - \sqrt{2 \log d}\}$: we have shown that this event holds with probabililty at least $1- \frac1d$ and that in this event, the lower bound of Theorem~\ref{thm:main} holds.
\end{rk}

\subsubsection{Upper bound} \label{subsub:upper-bound}

The upper bound stems from an application of~\cite[Theorem~$1$]{Tikhomirov17-0}. However, this result applies to empirical covariance matrices with centered columns, which is not the case of $\hat \Sigma_A$ except if $\E (\ell_A(X)^{1/2} (X - \mu_A)) = 0$, i.e., $\mu_A = \mu$ with
\begin{equation} \label{eq:def-mu}
	\mu = \frac{\E (\ell_A(X)^{1/2} X)}{\E \ell_A(X)^{1/2}}.
\end{equation}
In general this need not be the case, and so in order to center the columns, we consider in the sequel
\[ \hat \Sigma = \frac{1}{n} \sum_{i=1}^n \ell_A(X_i) (X_i - \mu) (X_i - \mu)^\top = \frac{1}{n} \sum_{i=1}^n S_i S_i^\top \ \text{ with } \ S_i = \ell_A(X_i)^{1/2}(X_i - \mu). \]
Then the columns $S_i$ are centered and so we will be able to use~\cite[Theorem~$1$]{Tikhomirov17-0} to control $\lVert \hat \Sigma - \E \hat \Sigma \rVert$. We will then transfer the result to $\lVert \hat \Sigma_A - \Sigma_A \rVert$ thanks for the following result.

\begin{lemma} \label{lemma:Sigma_A-Sigma}
	If Assumptions~\ref{ass:A}--\ref{ass:p} hold, then
	\[ \frac{\lVert \hat \Sigma_A - \Sigma_A \rVert}{\lVert \hat \Sigma_A \rVert} \lesssim \frac{\lVert \hat \Sigma - \Sigma \rVert}{\lVert \Sigma \rVert} + \lvert \hat 1 - 1 \rvert + \lVert \hat \mu_A - \mu_A \rVert. \]
\end{lemma}

First we prove that $\lVert \mu \rVert$ is bounded.

\begin{lemma} \label{lemma:mu}
	If Assumptions~\ref{ass:g} and~\ref{ass:p} hold, then $\lVert \mu \rVert$ is bounded.
\end{lemma}

\begin{proof}
	Starting from the definition~\eqref{eq:def-mu} of~$\mu$, we get $\mu = \E (\ell_A(X)^{1/2} P X) / \E \ell_A(X)^{1/2}$ using Lemma~\ref{lemma:W}, and so using the triangular inequality, Cauchy-Schwarz inequality and the fact that $\E \ell_A(X) = 1$, we get
	\[ \lVert \mu \rVert \leq \frac{\E^{1/2} \lVert P X \rVert^2}{\E \ell_A(X)^{1/2}} \lesssim \E^{1/2} \lVert PX \rVert^2 \]
	using Lemma~\ref{lemma:moments-ell} for the second inequality. Since $\E \lVert P X \rVert^2 = \tr \E PX(PX)^\top$, Lemma~\ref{lemma:W} gives
	\[ \E \lVert P X \rVert^2 = \tr(\Sigma_g - P_\perp + \mu_g \mu_g^\top) \]
	which is bounded by Assumptions~\ref{ass:A} and~\ref{ass:g}.
\end{proof}

\begin{proof} [Proof of Lemma~\ref{lemma:Sigma_A-Sigma}]
	Let in the sequel $\delta = \mu - \mu_A$ and $\hat \delta = \hat 1 \mu - \hat \mu_A$: then
	\begin{align*}
		(X_i - \mu_A) (X_i - \mu_A)^\top & = (X_i - \mu + \delta) (X_i - \mu + \delta)^\top\\
		& = (X_i - \mu) (X_i - \mu)^\top + \delta \delta^\top + (X_i- \mu) \delta^\top + \delta (X_i-\mu)^\top
	\end{align*}
	so that $\hat \Sigma_A = \hat \Sigma + \hat 1 \delta \delta^\top - \hat \delta \delta^\top - \delta \hat \delta^\top$. Since $\E \hat 1 = 1$ and $\E \hat \delta = \delta$, if $\Sigma = \E \hat \Sigma$ then we obtain
	\[ \E \hat \Sigma_A = \Sigma_A = \Sigma - \delta \delta^\top \ \text{ and } \ \hat \Sigma_A - \Sigma_A = \hat \Sigma - \Sigma + \varepsilon \]
	with
	\[ \varepsilon = \hat 1 \delta \delta^\top - \hat \delta \delta^\top - \delta \hat \delta^\top + \delta \delta^\top = (\hat 1 - 1) \delta \delta^\top + (\delta - \hat \delta) \delta^\top + \delta (\delta - \hat \delta)^\top. \]
	In particular,
	\[ \frac{\lVert \hat \Sigma_A - \Sigma_A \rVert}{\lVert \hat \Sigma_A \rVert} \leq \frac{\lVert \hat \Sigma - \Sigma \rVert}{\lVert \Sigma_A \rVert} + \frac{\lVert \varepsilon \rVert}{\lVert \Sigma_A \rVert} = \frac{\lVert \Sigma \rVert}{\lVert \Sigma_A \rVert} \frac{\lVert \hat \Sigma - \Sigma \rVert}{\lVert \Sigma \rVert} + \frac{\lVert \varepsilon \rVert}{\lVert \Sigma_A \rVert}. \]
	Since $\Sigma = \Sigma_A + \delta \delta^\top$, we get $\lVert \Sigma \rVert \leq \lVert \Sigma_A \rVert + \lVert \delta \delta^\top \rVert = \lVert \Sigma_A \rVert + \lVert \delta \rVert^2$ and so
	\[ \frac{\lVert \hat \Sigma_A - \Sigma_A \rVert}{\lVert \hat \Sigma_A \rVert} \leq \left(1 + \frac{\lVert \mu - \mu_A \rVert^2}{\lVert \Sigma_A \rVert} \right) \frac{\lVert \hat \Sigma - \Sigma \rVert}{\lVert \Sigma \rVert} + \frac{\lVert \varepsilon \rVert}{\lVert \Sigma_A \rVert}. \]
	Since $\lVert \Sigma_A \rVert \geq 1/\lVert \Sigma_A^{-1} \rVert$ and since $\lVert \mu_A \rVert$, $\lVert \Sigma_A^{-1} \rVert$ and $\lVert \mu \rVert$ are bounded by Lemmas~\ref{lemma:p} and~\ref{lemma:mu}, we get
	\[ \frac{\lVert \hat \Sigma_A - \Sigma_A \rVert}{\lVert \hat \Sigma_A \rVert} \lesssim \frac{\lVert \hat \Sigma - \Sigma \rVert}{\lVert \Sigma \rVert} + \lVert \varepsilon \rVert \]
	and so it remains to show that $\lVert \varepsilon \rVert \lesssim \lvert \hat 1 - 1 \rvert + \lVert \hat \mu_A - \mu_A \rVert$. Since $\lVert xy^\top \rVert \leq \lVert x \rVert \lVert y \rVert$ we have
	\[ \lVert \varepsilon \rVert \leq \lvert \hat 1 - 1 \rvert \lVert \delta \rVert^2 + 2 \lVert \delta \rVert \lVert \hat \delta - \delta \rVert \leq \lvert \hat 1 - 1 \rvert (\lVert \delta \rVert^2 + 2 \lVert \mu \rVert \lVert \delta \rVert) + 2 \lVert \delta \rVert \lVert \hat \mu_A - \mu_A \rVert. \]
	Since $\lVert \delta \rVert \leq \lVert \mu \rVert + \lVert \mu_A \rVert$ with $\lVert \mu \rVert$ and $\lVert \mu_A \rVert$ bounded, we get the result.
\end{proof}

\begin{lemma}
	If Assumptions~\ref{ass:A} and~\ref{ass:g} hold, then $\lVert P_\perp (\hat \mu_A - \mu_A) \rVert^2 \lesssim \hat 1 M$ with probability at least $1-\frac1d$.
\end{lemma}

\begin{proof}
	Let $L^2 = \frac{1}{n^2} \sum_{i=1}^n \ell_A(X_i)^2$ and $G = P_\perp \hat \mu_A / L$: since $P_\perp \mu_A = 0$ by Lemma~\ref{lemma:mu-Sigma-A}, we have $P_\perp (\hat \mu_A - \mu_A) = LG$ and since
	 \[ L^2 = \frac{1}{n^2} \sum_{i=1}^n \ell_A(X_i)^2 \leq \frac{1}{n} \max_{i=1, \ldots, n} \ell_A(X_i) \times \frac{1}{n} \sum_{i=1}^n \ell_A(X_i) = d^{-1} M \hat 1, \]
	we get $\lVert P_\perp (\hat \mu_A - \mu_A) \rVert^2 \leq \hat 1 M \frac{\lVert G \rVert^2}{d}$. Since $\lVert G \rVert \leq \E \lVert G \rVert + \sqrt{2 \log d}$ with probability at least $1-\frac{1}{d}$ by~\eqref{eq:concentration-G}, and since $\E \lVert G \rVert + \sqrt{2 \log d} \lesssim \sqrt d$, we get the result.
\end{proof}

\begin{rk} \label{rk:E2}
	The second event that we need to consider in order for the upper bound of Theorem~\ref{thm:main} to hold is thus, with the notation of the previous proof, $E_2 = \{\lVert G \rVert \leq \E \lVert G \rVert + \sqrt{2 \log d}\}$.
\end{rk}

In order to apply~\cite[Theorem~$1$]{Tikhomirov17-0}, we need to control moments of the one-dimensional projections of the $S_i$'s: we will use the following result.

\begin{lemma} \label{lemma:moments-S}
	If Assumptions~\ref{ass:A}--\ref{ass:p} hold, then for any $\gamma < \frac{2}{1-\lambda_1}$ we have
	\[ \sup_d \sup_{u \in \R^d: \lVert u \rVert = 1} \E \lvert S_1^\top u \rvert^\gamma < \infty. \]
\end{lemma}

\begin{proof}
	Let $1 < \alpha < \frac{2}{\gamma(1-\lambda_1)}$ and $\beta > 1$ with $1/\alpha + 1/\beta = 1$: by definition of $S_1$ we have
	\[ \E \lvert S_1^\top u \rvert^\gamma = \E \left( \ell_A(X)^{\gamma/2} \lvert (X - \mu)^\top u \rvert^\gamma \right) \]
	and so H\"older's inequality gives
	\[ \E \lvert S^\top u \rvert^\gamma \leq \left( \E^{1/\alpha} \ell_A(X)^{\gamma\alpha/2} \right) \times \left( \E^{1/\beta} \lvert (X - \mu)^\top u \rvert^{\gamma\beta} \right). \]
	By choice of $\gamma$ and $\alpha$, we have $\alpha \gamma/2-1 < \frac{\lambda_1}{1-\lambda_1}$ and so Lemma~\ref{lemma:moments-ell} implies that the first term in the right-hand side is bounded. For the second term, we have
	\[ \E^{1/\beta} \lvert (X - \mu)^\top u \rvert^{\gamma \beta} \leq \lVert X - \mu \rVert_{\psi_2}^\gamma \]
	with $\lVert G \rVert_{\psi_2} = \sup_{u: \lVert u \rVert = 1} \sup_{q \geq 1} q^{-1/2} \E^{1/q} \lvert G^\top u \rvert^q$ the sub-gaussian norm of a random vector $G$ (see for instance~\cite[Definition~$5.22$]{Vershynin12-0}). Thus in order to prove the result, it is enough to prove that $X - \mu$ has bounded sub-gaussian norm. Write $X = \Sigma_g^{1/2} Y + \mu_g$ with $Y$ standard Gaussian: then
	\[ \lVert X - \mu \rVert_{\psi_2} \leq \lVert \Sigma_g^{1/2} Y \rVert_{\psi_2} + \lVert \mu_g \rVert + \lVert \mu \rVert \]
	where we have used that $\lVert m \rVert_{\psi_2} = \lVert m \rVert$ for a constant $m$. From the definition of the sub-gaussian norm, it is clear that $\lVert \Sigma_g^{1/2} Y \rVert_{\psi_2} \leq \lVert \Sigma_g \rVert^{1/2} \lVert Y \rVert_{\psi_2}$. Since $\lVert \mu_g \rVert$ and $\lVert \Sigma_g \rVert$ are bounded by Assumption~\ref{ass:g}, $\lVert Y \rVert_{\psi_2}$ is bounded (see for instance~\cite[Lemma~$5.24$]{Vershynin12-0}) and $\lVert \mu \rVert$ is bounded by Lemma~\ref{lemma:mu}, we get the result.
\end{proof}

We can now prove the upper bound of Theorem~\ref{thm:main}. So consider $1 < a < \min(\frac{1}{1-\lambda_1}, 2)$. Then applying Lemma~\ref{lemma:moments-S} with $\gamma = 2a$, Theorem~$1.1$ in~\cite{Tikhomirov17-0} gives
\begin{equation} \label{eq:bound-Tikhomirov}
	\frac{\lVert \hat \Sigma - \Sigma \rVert}{\lVert \Sigma \rVert} \lesssim \frac{1}{n} \max_{i=1, \ldots, n} \langle S_i, \Sigma^{-1} S_i \rangle + \left( \frac{d}{n} \right)^{1-1/a} \log^4 \frac{n}{d} + \left( \frac{d}{n} \right)^{1-1/a}
\end{equation}
with probability at least $1 - \frac{1}{d}$. Since $\Sigma = \Sigma_A + \delta \delta^\top$, we have $\lambda_{\min}(\Sigma) \geq \lambda_{\min}(\Sigma_A)$ and so $\inf \lambda_{\min}(\Sigma) > 0$ and so
\begin{align*}
	\frac{\lVert \hat \Sigma - \Sigma \rVert}{\lVert \Sigma \rVert} & \lesssim \frac{1}{n} \max_{i=1, \ldots, n} \ell_A(X_i) \lVert X_i - \mu \rVert^2 + \left( \frac{d}{n} \right)^{1-1/a} \log^4 \frac{n}{d}\\
	& \lesssim \frac{1}{n} \max_{i=1, \ldots, n} \ell_A(X_i) \max_{i = 1, \ldots, n} \lVert X_i - \mu \rVert^2 + \left( \frac{d}{n} \right)^{1-1/a} \log^4 \frac{n}{d}.
\end{align*}
Thus in order to conclude the proof of the inequalities, it remains to show that
\[ \max_{i = 1, \ldots, n} \lVert X_i - \mu \rVert^2 \lesssim d + \log n \]
with probability at least $1 - \frac{1}{d}$. Write $X = \Sigma^{1/2}_g Y + \mu_g$ with $Y \sim f$ and let
\[ t = \lVert \Sigma_g \rVert \E \lVert Y \rVert + \lVert \mu_g \rVert + \lVert \mu \rVert + \sqrt{2 \log nd} \lVert \Sigma_g \rVert. \]
Then
\[ \P \left( \max_{i=1, \ldots, n} \lVert X_i - \mu \rVert \geq t \right) \leq n \P \left( \lVert X - \mu \rVert \geq t \right) \leq n \P \left( \lVert \Sigma_g \rVert \lVert Y \rVert + \lVert \mu_g \rVert + \lVert \mu \rVert \geq t \right) \]
and so from~\eqref{eq:concentration-G}, we get that $\max_{i \leq n} \lVert X_i - \mu \rVert \leq t$ with probability at least $1 - \frac{1}{d}$. Since $t^2 \lesssim d + \log n$ we get the result.

\begin{rk} \label{rk:E3-E4}
	We define $E_3$ the event of~\cite[Theorem~$1$]{Tikhomirov17-0} where~\eqref{eq:bound-Tikhomirov} holds, as well as $E_4 = \{ \max_{i \leq n} \lVert X_i - \mu \rVert \leq t\}$.
\end{rk}

\subsubsection{Proof of the phase transition~\eqref{eq:PT}} \label{subsub:phase-transition}

The lower bound $M - 1 \lesssim \frac{\lVert \hat \Sigma_A - \Sigma_A \rVert}{\lVert \Sigma_A \rVert}$ immediatly implies that $\frac{\lVert \hat \Sigma_A - \Sigma_A \rVert}{\lVert \Sigma_A \rVert} \Rightarrow +\infty$ if $M \Rightarrow +\infty$.

So assume now that $M \Rightarrow 0$, and let us prove that $\frac{\lVert \hat \Sigma_A - \Sigma_A \rVert}{\lVert \Sigma_A \rVert} \Rightarrow 0$ as well. First, note that since $\E \ell_A(X) = 1$, $\ell_A(X)^{1/2}$ is uniformly integrable: since $\inf \E \ell_A(X)^{1/2} > 0$ by Lemma~\ref{lemma:moments-ell}, this prevents $\ell_A(X) \Rightarrow 0$. In particular, $\max_{i \leq n} \ell_A(X_i)$ does not vanish (in law), and so the condition
\[ M = \frac{d}{n} \max_{i = 1, \ldots, n} \ell_A(X_i) \Rightarrow 0 \]
implies that, necessarily, $n \gg d$. Therefore, the therm $(d/n)^{1-1/a} \log^4(n/d)$ vanishes, and since we assume that $d \gg \log n$, so does the term $\log n / d$. Since $n \to \infty$, the terms $\lvert \hat 1 - 1 \rvert$ and $\lVert P(\hat \mu_A - \mu_A) \rVert$ also vanish by Lemma~\ref{lemma:error-terms}, so that the upper bound can be rewritten as
\[ \frac{\lVert \hat \Sigma_A - \Sigma_A \rVert}{\lVert \Sigma_A \rVert} \leq (1 + o(1)) M + o(1). \]
Since $M \Rightarrow 0$ as well, the upper does indeed vanish, which proves the desired result.

\subsection{Proof of Lemma~\ref{lemma:error-terms}} \label{sub:proof-lemma}

Let $(u_1, \ldots, u_J)$ be an orthonormal basis of $\SA \oplus \Sg \oplus \mathrm{span}(\mu_g)$, and for $j = 0, \ldots, J$, define $\varphi_j: \R^d \to \R$ as follows: $\varphi_0(x) = 1$ and $\varphi_j(x) = x^\top u_j$ for $j = 1, \ldots, J$. Consider finally $\delta_j$ the error made when estimating $\E \varphi_j(Y_A)$ by importance sampling:
\[ \delta_j = \left \lvert \frac{1}{n} \sum_{i=1}^n \ell_A(X_i) \varphi_j(X_i) - \E \varphi_j(Y_A) \right \rvert. \]

Then by definition, we have $\lvert \hat 1 - 1 \rvert = \delta_0$. Moreover, since
\[ P(\hat \mu_A - \mu_A) = \sum_{j = 1}^J \left( \frac{1}{n} \sum_{i=1}^n \ell_A(X_i) \langle X_i, u_j \rangle - \langle \mu_A, u_j \rangle \right) u_j \]
and the L\textsubscript2 norm is smaller than the L\textsubscript1 norm, we have
\[ \left \lVert P(\hat \mu_A - \mu_A) \right \rVert \leq \sum_{j = 1}^J \left \lvert \frac{1}{n} \sum_{i=1}^n \ell_A(X_i) X_i^\top u_j - \mu^\top_A u_j \right \rvert = \sum_{j=1}^J \delta_j. \]
Thus, in order to prove Lemma~\ref{lemma:error-terms}, it is enough to find a constant $\gamma > 0$ such that that for each $j$, the event $\{ \delta_j \leq \frac{1}{n^\gamma} \}$ holds with probability at least $1 - \frac{1}{(J+1) n^\gamma}$. Indeed, in this case, the event $\cap_{j=0}^J \{ \delta_j \leq \frac{1}{n^\gamma} \}$ holds with probability at least $1 - \frac{1}{n^\gamma}$, and since we have seen that $\lvert \hat 1 - 1 \rvert + \lVert P(\hat \mu_A - \mu_A) \rVert \leq \sum_{j=0}^J \delta_j$, in this event we have $\lvert \hat 1 - 1 \rvert + \lVert P(\hat \mu_A - \mu_A) \rVert \lesssim \frac{1}{n^\gamma}$ as desired.

Consider any $0 < \varepsilon < \min(1, \frac{\lambda_1}{1 - \lambda_1})$ and let $\gamma =  \varepsilon / 8$: let us show that $\P(\delta_j \geq n^{-\gamma}) \lesssim n^{-\gamma}$. By Markov inequality, we have $\P(\delta_j \geq \frac{1}{n^\gamma}) \leq n^\gamma \E \delta_j$. Moreover, Theorem~$1.1$ in~\cite{Chatterjee18:0} shows that
\[ \E \delta_j \leq \E^{1/2} \varphi_j(Y_A)^2 \times \left( \frac{e^{D(f|_A, g)/4}}{n^{1/4}} + 2 \P^{1/2} \left( \ell_A(Y_A) \geq \sqrt{n e^{D(f|_A, g)}} \right) \right). \]
Let $0 < \varepsilon < \min(1, \frac{\lambda_1}{1 - \lambda_1})$: again by Markov inequality, we have
\[ \P \left( \ell_A(Y_A) \geq \sqrt{n e^{D(f|_A, g)}} \right) \leq \frac{\E \ell_A(Y_A)^\varepsilon}{n^{\varepsilon/2} e^{\varepsilon D(f|_A, g)/2}} \leq \frac{\E \ell_A(Y_A)^\varepsilon}{n^{\varepsilon/2}} = \frac{\E \ell_A(X)^{1+\varepsilon}}{n^{\varepsilon/2}}. \]
Since $D(f|_A, g)$ and $\E \ell_A(X)^{1+\varepsilon}$ are bounded by Lemmas~\ref{lemma:p} and~\ref{lemma:moments-ell}, respectively, we finally obtain
\[ \E \delta_j \lesssim \E^{1/2} \varphi_j(Y_A)^2 \times \left( \frac{1}{n^{1/4}} + \frac{1}{n^{\varepsilon / 4}} \right) \lesssim \E^{1/2} \varphi_j(Y_A)^2 \frac{1}{n^{\varepsilon / 4}}. \]
We have $\E \varphi_0(Y_A) = 1$ while for $j = 1, \ldots, J$, $\E \varphi_j(Y_A)^2 = \E \langle u_j, Y_A \rangle^2 \lesssim \E \langle u_j, Y \rangle^2 = 1$ as a consequence of Assumption~\ref{ass:p}. Therefore, we have
\[ \P \left( \delta_j \geq \frac{1}{n^\gamma} \right) \lesssim \frac{1}{n^{\varepsilon/4 - \gamma}}. \]
Since $\varepsilon / 4 - \gamma = \gamma$ this concludes the proof.

\subsection{Proof of Propositions~\ref{prop:1/lambda},~\ref{prop:hyperplane-<} and~\ref{prop:hyperplane->}}  \label{sub:proof-prop}

These propositions are concerned with whether $M = \frac{d}{n} \max_{i \leq n} \ell_A(X_i) \Rightarrow 0$ or to $+\infty$. Let $x > 0$: since
\[ \P(M < x) = \P \left( \ell_A(X) < xn/d \right)^n = \exp \left( n \log \left( 1 - \P \left( \ell_A(X) \geq xn/d \right) \right) \right), \]
it is clear that
\[ M \Rightarrow \left\{ \begin{array}{ll} 0 & \text{ if } n \P(\ell_A(X) \geq x n / d) \to 0 \text{ for every } x > 0,\\ +\infty & \text{ if } n \P(\ell_A(X) \geq x n / d) \to \infty \text{ for every } x > 0. \end{array} \right. \]
Thus the problem reduces to studying the tail behavior of $\ell_A(X)$. Moreover, recalling that $\ell_A = \ell \xi_A / p$ with $\ell = f/g$, we have
\begin{equation} \label{eq:tail-ell_A}
	\P(\ell_A(X) \geq x n / d) = \P(\ell(X) \geq p x n / d, X \in A).
\end{equation}
The next lemma shows that if $\inf_d \P(Y \in A) > 0$ (i.e., Assumption~\ref{ass:p} holds), then the event~$A$ remains of positive probability under $g$ provided that Assumption~\ref{ass:g} also holds.

\begin{lemma} \label{lemma:q}
	If Assumptions~\ref{ass:g} and~\ref{ass:p} hold, then $\inf_d \P(X \in A) > 0$.
\end{lemma}

\begin{proof}
	We have
	\[ \P(X \in A) = \E \xi_A(X) = \E \left( \ell(Y)^{-1} \xi_A(Y) \right) = p \E \left( \ell(Y_A)^{-1} \right) = \E \left( \ell_A(Y_A)^{-1} \right). \]
	This can be rewritten as $\P(X \in A) = \E \left[ \exp \left( -\log \ell_A(Y_A) \right) \right]$ and so Jensen's inequality for the exponential function gives $\P(X \in A) \geq e^{-D(f|_A, g)}$ which implies $\inf_d \P(X \in A) > 0$ in view of Lemma~\ref{lemma:p}.
\end{proof}

Let us pursue from~\eqref{eq:tail-ell_A}. If $y = \Sigma_g^{-1/2} (x - \mu_g)$, then
\begin{align*}
	\ell(x) & = \ell(\Sigma_g^{1/2} y + \mu_g)\\
	& = \lvert \Sigma_g \rvert^{1/2} \exp \left( - \frac{1}{2} \lVert \Sigma_g^{1/2} y + \mu_g \rVert^2 + \frac{1}{2} \lVert y \rVert^2 \right)\\
	& = \exp \left( \frac{1}{2} y^\top (I - \Sigma_g) y - \mu^\top_g \Sigma^{1/2}_g y - \frac{1}{2} \lVert \mu_g \rVert^2 + \frac{1}{2} \log \lvert \Sigma_g \rvert \right)
\end{align*}
and so if $\gamma = \exp( \frac{1}{2} \lVert \mu_g \rVert^2 - \frac{1}{2} \log \lvert \Sigma_g \rvert)$ and $z = 2 \Sigma_g^{1/2} \mu_g$,~\eqref{eq:tail-ell_A} can be rewritten as
\begin{equation} \label{eq:formula-tail-ell_A}
	\P(\ell_A(X) \geq x n / d) = \P \left( Y^\top (I - \Sigma_g) Y - z^\top Y \geq 2 \log(\gamma pxn/d), \Sigma^{1/2}_g Y + \mu_g \in A \right).
\end{equation}

We can now consider each proposition.

\subsubsection{Proof of Proposition~\ref{prop:1/lambda}}

Assume that $n \gg d^\kappa$ for some $\kappa > 1/\lambda_1$: we have to prove that $M \Rightarrow 0$, i.e., that $n \P(\ell_A(X) \geq nx/d) \to 0$ for every $x > 0$. Let $\varepsilon = \frac{1}{1-\kappa}$: then
\[ n \P \left( \ell_A(X) \geq n x / d \right) \leq \frac{1}{x^{1+\varepsilon}} \E \ell_A(X)^{1+\varepsilon} \frac{n d^{1+\varepsilon}}{n^{1+\varepsilon}} \lesssim \left( \frac{d^\kappa}{n} \right)^\varepsilon, \]
with the second inequality coming from Lemma~\ref{lemma:moments-ell} since $\varepsilon < \frac{\lambda_1}{1-\lambda_1}$, which implies the boundedness of $\E \ell_A(X)^{1+\varepsilon}$. This shows that $M \Rightarrow 0$.

Assume now that $n \ll d^\kappa$ for some $1 < \kappa < 1/\lambda_1$, and that $\Sg \subset {\SA}_\perp$: we have to show that $M \Rightarrow +\infty$, i.e., that $n \P(\ell_A(X) \geq x n / d) \to \infty$ for every $x > 0$. For simplicity we will prove the result for $n = d^\kappa$ for some $1 < \kappa < 1 / \lambda_1$. Note that $\cov(P_{\Sg} X, P_{\SA} X) = P_{\Sg} \Sigma_g P_{\SA} = 0$: indeed, according to Lemma~\ref{lemma:mu-Sigma-A}, for any $y \in \R^d$ we have $\Sigma_g P_{\SA} y \in \SA$ and since $\Sg \subset {\SA}_\perp$, this implies that $P_{\Sg} (\Sigma_g P_{\SA} y) = 0$. Thus, since $X$ is Gaussian, $P_{\Sg} X$ and $P_{\SA} X$ are independent. Inspecting the proof of Lemma~\ref{lemma:W}, we see that $\ell = \ell \circ P_{\Sg}$ and $\xi_A = \xi_A \circ P_{\SA}$: thus $\ell_A(X)$ and $\xi_A(X)$ are independent, and so
\[ \P(\ell_A(X) \geq x n / d) = \P(\ell(X) \geq p x n / d) \P(X \in A). \]
In view of Lemma~\ref{lemma:q}, we thus have to show that $n \P(\ell(X) \geq p x n / d) \to \infty$ or, in view of~\eqref{eq:formula-tail-ell_A}, that
\[ n \P \left( Y^\top (I - \Sigma_g) Y - z^\top Y \geq 2 \log(\gamma pxn/d) \right) \to \infty. \]
If $N_i = Y^\top v_i$, we can write
\[ Y^\top (I - \Sigma_g) Y - z^\top Y = \sum_{i=1}^r \left( (1 - \lambda_i) N_i^2 - \langle z, v_i \rangle N_i \right) = (1 - \lambda_1) (N_1 - a)^2 + Z \]
with $a = z^\top v_1 / (2(1-\lambda_1))$ and where $Z$ only depends on the $N_i$'s for $i \neq 1$. Thus $N_1$ and $Z$ are independent, so that for any $K \geq 0$ we have
\begin{multline*}
	\P \left( Y^\top (I - \Sigma_g) Y - z^\top Y \geq 2 \log (\gamma p x n / d) \right)\\
	\geq \P \left( (1 - \lambda_1) (N_1 - a)^2 \geq 2 \log (K \gamma p x n / d) \right) \times \P \left( Z \geq -2 \log K \right)
\end{multline*}
It is easy to check that there exists $K > 0$ such that $\inf_d \P(Z \geq -2 \log K) > 0$, so that
\[ \P \left( Y^\top (I - \Sigma_g) Y - z^\top Y \geq 2 \log (\gamma p x n / d) \right) \gtrsim \P \left( (1 - \lambda_1) (N_1 - a)^2 \geq 2 \log (K \gamma p x n / d) \right) \]
and finally, using $n = d^\kappa$,
\[ \P \left( Y^\top (I - \Sigma_g) Y - z^\top Y \geq 2 \log (\gamma p x n / d) \right) \gtrsim \P \left( N_1 \geq a + \sqrt{b + c \log d} \right) \]
with $b = 2 \log(K \gamma p x)/(1-\lambda_1)$ and $c = 2 (\kappa-1)/(1-\lambda_1)$. Using the classical first-order expansion $\P(N_1 \geq y) \sim (2\pi)^{-1/2} y^{-1} e^{-y^2/2}$ for large $y$, we get since $a$ and $b$ are bounded above and below,
\[ n \P \left( N_1 \geq a + \sqrt{b + c \log d} \right) \sim \frac{n}{\sqrt{2\pi c \log d}} \exp \left( -\frac{1}{2} \left( a^2 + b + c \log d + 2a \sqrt{b + c \log d} \right) \right). \]
Since $\sqrt{b + c \log d)} = \sqrt{c \log d} + o(1)$ this gives
\begin{align*}
	n \P \left( N_1 \geq a + \sqrt{b + c \log d} \right) & \sim \frac{C n}{\sqrt{\log d}} \exp \left( -\frac{1}{2} \left( c \log d + 2a \sqrt{c \log d} \right) \right)\\
	& = C \exp \left( \left( \kappa - \frac{c}{2} \right) \log d - a \sqrt{c \log d} - \frac{1}{2} \log \log d \right)
\end{align*}
with $C = (2\pi c)^{-1/2} e^{-(a^2 + b)/2}$. Since
\[ \kappa - \frac{c}{2} = \kappa - \frac{\kappa-1}{1-\lambda_1} = \frac{1-\kappa \lambda_1}{1-\lambda_1} > 0 \]
we get $n \P \left( N_1 \geq a + \sqrt{b + c \log d} \right) \to \infty$ which, in view of the above arguments, implies that $M \Rightarrow +\infty$ as desired.

\subsubsection{Proof of Proposition~\ref{prop:hyperplane-<}}

We now consider the case $A = \{x: \lvert \langle x, u \rangle \rvert \leq K \}$, $\Sigma_g = I + (\lambda_1 - 1) uu^\top$ and $\mu_g = 0$. It is clear that Assumptions~\ref{ass:A}--\ref{ass:p} are satisfied with $\Sg = \SA = \mathrm{span}(u)$. With the notation of~\eqref{eq:formula-tail-ell_A}, we have
\[ \gamma = \lambda_1^{-1/2}, \ z = 0, \ Y^\top (I - \Sigma_g) Y = (1-\lambda_1) \langle Y, u \rangle^2 \]
and
\[ \Sigma_g^{1/2} Y + \mu_g \in A \Leftrightarrow \left \lvert \left \langle Y + (\lambda_1^{1/2} - 1) \langle Y, u \rangle u, u \right \rangle \right \rvert \leq K \Leftrightarrow \lambda_1^{1/2} \left \lvert \langle Y, u \rangle \right \rvert \leq K \]
so that
\begin{align*}
	\P \left( \ell_A(X) \geq xn/d \right) & = \P \left(\lambda_1^{-1} \langle Y, u \rangle^2 \geq 2 \log (\lambda_1^{-1/2} pxn/d), \lambda_1 \left \lvert \langle Y, u \rangle \right \rvert \leq K \right)\\
	& = \P \left(2 (1-\lambda_1)^{-1} \log (\lambda_1^{-1/2} pxn/d) \leq \langle Y, u \rangle^2 \leq \lambda_1^{-1} K^2 \right).
\end{align*}
Since $K = \sqrt{2 \alpha \lambda_1 \log n}$ with $\alpha = \frac{1 - 1/\kappa_*}{1-\lambda_1}$, we have $\lambda_1^{-1} K^2 = 2 \frac{1 - 1/\kappa_*}{1-\lambda_1} \log n$ and so
\[ \P \left( \ell_A(X) \geq xn/d \right) = \P \left(\gamma' + \log \left( \frac{n}{d} \right)^{1/(1-\lambda_1)} \leq \frac{1}{2} \langle Y, u \rangle^2 \leq \log n^{(1-1/\kappa_*)/(1-\lambda_1)} \right) \]
with $\gamma' = (1-\lambda_1)^{-1} \log(\lambda_1^{-1/2} px)$ which is bounded. If $n^{1/\kappa_*} \ll d$, the lower bound in the probability is, for $n$ large enough, smaller than the lower bound which makes the probability equal to $0$. Thus, this shows the desired result when $n \ll d^{\kappa_*}$.

So let us now consider $n \gg d^\kappa$ for some $\kappa > \kappa_*$: for simplicity, we consider $n = d^\kappa$ with $\kappa > \kappa_*$. Using again $\P(N \geq x) \sim (2\pi x^2)^{-1/2} e^{-x^2/2}$, we obtain
\[ \P \left(\frac{1}{2} \langle Y, u \rangle^2 \geq \gamma' + \log \left( \frac{n}{d} \right)^{1/(1-\lambda_1)} \right) \sim c \left( \sqrt{\log(n/d)} \left( \frac{n}{d} \right)^{1/(1-\lambda_1)} \right)^{-1} \]
and
\[ \P \left(\frac{1}{2} \langle Y, u \rangle^2 \geq \log n^{(1-1/\kappa_*) / (1-\lambda_1)} \right) \sim C \left( \log(n) n^{(1-1/\kappa_*)/(1-\lambda_1)} \right)^{-1} \]
for some constants $c, C \in (0,\infty)$. Because $n = d^\kappa$ with $\kappa > \kappa_*$, the previous first-order expansions imply that
\[ \P \left(\frac{1}{2} \langle Y, u \rangle^2 \geq \log n^{(1-1/\kappa_*) / (1-\lambda_1)} \right) \ll \P \left(\frac{1}{2} \langle Y, u \rangle^2 \geq \gamma' + \log \left( \frac{n}{d} \right)^{1/(1-\lambda_1)} \right) \]
and thus that
\[ n \P \left( \ell_A(X) \geq xn/d \right) \sim n \P \left(\frac{1}{2} \langle Y, u \rangle^2 \geq \gamma' + \log \left( \frac{n}{d} \right)^{1/(1-\lambda_1)} \right) \]
which diverges to $+\infty$.

\subsubsection{Proof of Proposition~\ref{prop:hyperplane->}}

We now consider the case $A = \{x: \langle x, u \rangle \geq K \}$ for some $u \in \R^d$ with $\lVert u \rVert = 1$, $K > 0$ independent of $d$, $\mu_g = 0$, $\Sigma_g = I + (\lambda_1 - 1) uu^\top$ and $n \ll d^\kappa$ for some $\kappa < 1 / \lambda_1$: we have to prove that $M \Rightarrow \infty$. Unfolding the same arguments as above, we get
\[ \P \left( \ell_A(X) \geq xn/d \right) = \P \left(\langle Y, u \rangle^2 \geq \max \left( 2 \lambda_1 \log (\lambda_1^{-1/2} pxn/d), \lambda_1^{-1} K^2 \right) \right) \]
and so the first-order expansion $\P(N \geq x) \sim (2\pi x^2)^{-1/2} e^{-x^2/2}$ gives the result.

\end{document}